\documentclass[nosumlimits,twoside]{amsart}
\usepackage{amsfonts, amsmath, amssymb}
\usepackage{graphicx}
\usepackage{float}
\usepackage{srcltx}
\usepackage[all]{xy}
\usepackage{version}
\usepackage[T1]{fontenc}
\usepackage{xcolor}
\usepackage{enumerate}
\usepackage[normalem]{ulem}
\usepackage{bm}
\usepackage{latexsym}
\usepackage[2emode]{psfrag}
\usepackage{yhmath}
\usepackage{array}
\usepackage{dsfont}
\usepackage{mathrsfs}
\usepackage{tikz-cd}
\usepackage{comment}

\newtheorem{theorem}{\sc Theorem}[subsection]
\newtheorem{proposition}[theorem]{\sc Proposition}

\newtheorem{lemma}[theorem]{\sc Lemma}
\newtheorem{corollary}[theorem]{\sc Corollary}
\theoremstyle{definition}
\newtheorem{definition}[theorem]{\sc Definition}

\newtheorem{example}[theorem]{\sc Example}

\theoremstyle{remark}
\newtheorem{remark}[theorem]{\sc Remark}

\newenvironment{invisible}{{\noindent\sc \colorbox{yellow}{Invisible:}\;}\color{gray}}{\medskip}
\excludeversion{invisible}

\allowdisplaybreaks
\excludeversion{invisible?}
\excludeversion{proof?}

\newcommand{\Cc}{\mathcal{C}}
\newcommand{\Dd}{\mathcal{D}}
\newcommand{\Mm}{\mathcal{M}}
\newcommand{\mm}{\mathfrak{M}}

\newcommand{\Rr}{\mathcal{R}}

\newcommand{\ot}{\otimes}

{
\left\{\begin{aligned}}
{\end{aligned}
\right.
}

\begin{document}

\title[\tiny Yetter--Drinfeld post-Hopf algebras and Yetter--Drinfeld relative Rota--Baxter operators]{Yetter--Drinfeld post-Hopf algebras and Yetter--Drinfeld relative Rota--Baxter operators}
\author{Andrea Sciandra}
\begin{abstract}\noindent
    Recently, Li, Sheng and Tang introduced post-Hopf algebras and relative Rota--Baxter operators (on cocommutative Hopf algebras), providing an adjunction between the respective categories under the assumption that the structures involved are cocommutative. We introduce Yetter--Drinfeld post-Hopf algebras, which become usual post-Hopf algebras in the cocommutative setting. In analogy with the correspondence between cocommutative post-Hopf algebras and cocommutative Hopf braces, the category of Yetter--Drinfeld post-Hopf algebras is isomorphic to the category of Yetter--Drinfeld braces introduced by the author in a joint work with D. Ferri. This allows to explore the connection with matched pairs of actions and provide examples of Yetter--Drinfeld post-Hopf algebras. Moreover, we prove that the category of Yetter--Drinfeld post-Hopf algebras is equivalent to a subcategory of Yetter--Drinfeld relative Rota--Baxter operators. The latter structures coincide with the inverse maps of Yetter--Drinfeld 1-cocycles introduced by the author and D. Ferri, and generalise bijective relative Rota--Baxter operators on cocommutative Hopf algebras. Hence the previous equivalence passes to cocommutative post-Hopf algebras and bijective relative Rota--Baxter operators. Once the surjectivity of the Yetter--Drinfeld relative Rota--Baxter operators is removed, the equivalence is replaced by an adjunction and one can recover the result of Li, Sheng and Tang in the cocommutative case.
\end{abstract}
\address{%
\parbox[b]{0.9\linewidth}{University of Turin, Department of Mathematics ``G.\@ Peano'',\\ via
 Carlo Alberto 10, 10123 Torino, Italy.}}
 \email{andrea.sciandra@unito.it}
\maketitle
\tableofcontents

\section{Introduction}
\noindent Post-Lie algebras have been introduced in \cite{V} and have found several interesting applications \cite{CEO,MuLu,BK}, just to name a few; while Rota--Baxter operators on Lie algebras first appeared in \cite{Ku} and are now studied in different fields, such as quantum field theory \cite{CK,Kreimer}. Post-Lie algebras and Rota--Baxter operators are strictly related: given a post-Lie algebra, the identity morphism is a Rota--Baxter operator on the subadjacent Lie algebra and, vice versa, one can induce a post-Lie algebra structure on the domain of a Rota--Baxter operator. In \cite{YYT} Li, Sheng and Tang introduced post-Hopf algebras and relative Rota--Baxter operators on cocommutative Hopf algebras, the latter generalising Rota--Baxter operators given in \cite{Go} (where the action involved is the adjoint one). The space of primitive elements of a post-Hopf algebra results to be a post-Lie algebra and a relative Rota--Baxter operator on a cocommutative Hopf algebra restricts to a Rota--Baxter operator between the Lie algebras of primitive elements. Moreover, in \cite{YYT}, the aforementioned connection between post-Lie algebras and Rota--Baxter operators is extended to cocommutative post-Hopf algebras and relative Rota--Baxter operators on cocommutative Hopf algebras, providing an adjunction between the respective categories. In \cite{YYT2} (which is the arXiv version of \cite{YYT}), a correspondence between cocommutative post-Hopf algebras and cocommutative Hopf braces is exhibited, then further investigated \cite{HLTL} in terms of relative Rota--Baxter operators. Hopf braces were introduced by Angiono, Galindo and Vendramin in \cite{AGV} as the Hopf-theoretic version of skew braces, introduced in \cite{GV}. Moreover, cocommutative Hopf braces are shown to be equivalent to matched pairs of cocommutative Hopf algebras. In \cite{FS} the author and D. Ferri proved that, removing the cocommutativity hypothesis, matched pairs result to be equivalent to the more general Yetter--Drinfeld braces. The latter structures coincide with Hopf braces if the cocommutativity is assumed.
\vskip 10pt

\noindent The purpose of this paper is to generalise post-Hopf algebras and relative Rota--Baxter operators in a not-cocommutative setting, keeping in mind the connection with Yetter--Drinfeld braces. 

More precisely, the content of the paper is the following. In Section \ref{sec:preliminaries} we recall the notions and results that motivate this work and are frequently used throughout the paper. In Section \ref{sec:YetterDrinfeldpostHopf} we introduce Yetter--Drinfeld post-Hopf algebras, which coincide with post-Hopf algebras in the cocommutative case, and we prove that the category of Yetter--Drinfeld post-Hopf algebras is isomorphic to the category of Yetter--Drinfeld braces (and thus to the category of matched pairs of actions). Explicitly, a Yetter--Drinfeld post-Hopf algebra is the datum $(H,\cdot,1,\Delta,\epsilon,S,\rightharpoonup)$ where $(H,\cdot,1)$ is an algebra, $(H,\Delta,\epsilon)$ is a coalgebra, $S$ is the convolution inverse of the morphism $\mathrm{Id}_{H}$ and $\rightharpoonup$ is a morphism of coalgebras satisfying some technical hypotheses. Then, one can build the subadjacent Hopf algebra $H_{\rightharpoonup}:=(H,\bullet_{\rightharpoonup},1,\Delta,\epsilon,S_{\rightharpoonup})$ so that $(H,\cdot,1,\Delta,\epsilon,S)$ results to be an object in $\mathrm{Hopf}(^{H_{\rightharpoonup}}_{H_{\rightharpoonup}}\mathcal{YD})$ and $(H,\cdot,\bullet_{\rightharpoonup},1,\Delta,\epsilon,S,S_{\rightharpoonup})$ becomes a Yetter--Drinfeld brace. A Yetter--Drinfeld post-Hopf algebra induces a post-Lie algebra structure on its space of primitive elements. At the end of the section we provide examples of Yetter--Drinfeld post-Hopf algebras coming from Yetter--Drinfeld braces. Finally, in Section \ref{sec:YetterDrinfeldRotaBaxter} we introduce Yetter--Drinfeld relative Rota--Baxter operators; these generalise relative Rota--Baxter operators on cocommutative Hopf algebras and coincide with the inverse maps of Yetter--Drinfeld 1-cocycles introduced in \cite{FS}. As in the cocommutative case we can restrict Yetter--Drinfeld relative Rota--Baxter operators to primitive elements and group-like elements, obtaining Rota--Baxter operators on Lie algebras and groups, respectively. We show that a subcategory of Yetter--Drinfeld relative Rota--Baxter operators is equivalent to the category of Yetter--Drinfeld post-Hopf algebras: given a  Yetter--Drinfeld post-Hopf algebra $(H,\rightharpoonup)$ the identity morphism $\mathrm{Id}_{H}:H\to H_{\rightharpoonup}$ is a Yetter--Drinfeld relative Rota--Baxter operator and, vice versa, one can induce a structure of Yetter--Drinfeld post-Hopf algebra on the domain (or codomain) of a Yetter--Drinfeld relative Rota--Baxter operator.
This equivalence passes to cocommutative post-Hopf algebras and bijective relative Rota--Baxter operators on cocommutative Hopf algebras. Once the surjectivity of the Yetter–Drinfeld relative Rota–Baxter operators is removed, the equivalence is replaced by an adjunction and one can recover, in the cocomutative case, the result given in \cite{YYT}. Then, the aforementioned correspondence between post-Lie algebras and Rota--Baxter operators, generalised in \cite{YYT} for cocommutative Hopf algebras, is now given in full generality. \vskip 10pt

\noindent\textit{Notations and conventions}. 
We denote by $\Bbbk$ an arbitrary field and all vector spaces will be $\Bbbk$-vector spaces. By a linear map we mean a $\Bbbk$-linear map and the unadorned tensor product $\otimes$ is the one of $\Bbbk$-vector spaces. Algebras over $\Bbbk$ will be associative and unital and coalgebras over $\Bbbk$ will be coassociative and counital. We use symbols like $\bullet$ and $\cdot$ for the multiplication of an algebra; equivalently, the multiplication will be denoted by $m$, $m_{\cdot}$, $m_{\bullet}$. The unit of an algebra $A$ will be denoted by $1$ or $u:\Bbbk\to A$. The comultiplication and the counit of a coalgebra $C$ will be denoted by $\Delta$ and $\epsilon$, respectively. We will use Sweedler’s notation for calculations involving the coproduct, i.e. we write $\Delta(c)=c_{1}\otimes c_{2}$ for all $c\in C$, omitting the summation. The antipode of a Hopf algebra will be denoted by $S$ or $T$.

\section{Preliminaries}\label{sec:preliminaries}
\noindent In this section we recall the notions and results that we are going to use throughout the paper. For more details about Hopf algebras and category theory we refer the reader to \cite{Sw,Majid2,Rad} and \cite{Bor}, respectively. \medskip

\noindent\textbf{Post-Lie algebras and relative Rota--Baxter operators}. The notion of post-Lie algebra was introduced in \cite{V} and it has found several interesting applications, such as \cite{CEO,MuLu,BK}. Recall that a \textit{post-Lie algebra} $(\mathfrak{g},[\cdot,\cdot],\rightharpoonup)$ is the datum of a Lie algebra $(\mathfrak{g},[\cdot,\cdot])$ and a linear map $\rightharpoonup:\mathfrak{g}\otimes\mathfrak{g}\to\mathfrak{g}$ such that, for all $x,y,z\in\mathfrak{g}$, the following equalities hold:
\begin{equation}\label{comp.act-brac}
      x\rightharpoonup[y,z]=[x\rightharpoonup y,z]+[y,x\rightharpoonup z],
\end{equation}
\begin{equation} \label{postLiecondition2}
\big([x,y]+(x\rightharpoonup y)-(y\rightharpoonup x)\big)\rightharpoonup z=\big(x\rightharpoonup(y\rightharpoonup z)\big)-\big(y\rightharpoonup(x\rightharpoonup z)\big).
\end{equation}
If $(\mathfrak{g},[\cdot,\cdot])$ is abelian, $(\mathfrak{g},[\cdot,\cdot],\rightharpoonup)$ is called a pre-Lie algebra \cite{CL}. More precisely, a \textit{pre-Lie algebra} $(\mathfrak{g},\rightharpoonup)$ is just a vector space $\mathfrak{g}$ equipped with a linear map $\rightharpoonup:\mathfrak{g}\otimes\mathfrak{g}\to\mathfrak{g}$ such that
\[
\big((x\rightharpoonup y)-(y\rightharpoonup x)\big)\rightharpoonup z=\big(x\rightharpoonup(y\rightharpoonup z)\big)-\big(y\rightharpoonup(x\rightharpoonup z)\big).
\]
To any post-Lie algebra $(\mathfrak{g},[\cdot,\cdot],\rightharpoonup)$ one can associate a Lie algebra $\mathfrak{g}_{\rightharpoonup}:=(\mathfrak{g},[\cdot,\cdot]_{\rightharpoonup})$, called the \textit{subadjacent Lie algebra}, where the Lie bracket $[\cdot,\cdot]_{\rightharpoonup}$ is defined by
\[
[x,y]_{\rightharpoonup}:=(x\rightharpoonup y)-(y\rightharpoonup x)+[x,y],
\]
so that \eqref{comp.act-brac}-\eqref{postLiecondition2} equivalently mean that $\mathfrak{g}\to\mathfrak{gl}(\mathfrak{g}),x\mapsto\big(x\rightharpoonup(-)\big)$ is an action of the Lie algebra $(\mathfrak{g},[\cdot,\cdot]_{\rightharpoonup})$ on the Lie algebra $(\mathfrak{g},[\cdot,\cdot])$. Recently, many studies on the universal enveloping algebra of a pre-Lie algebra and a post-Lie algebra were conducted, see e.g. \cite{OG,ELM}. \medskip

We also mention that the corresponding notion of \textit{post-group} was introduced in \cite{BGST}. It consists of a group $(G,\cdot)$ equipped with a map $\rightharpoonup:G\times G\to G$ such that, for each $a\in G$, the map $a\rightharpoonup(-):G\to G$ is an automorphism of the group $(G,\cdot)$, that is
\[
a\rightharpoonup(b\cdot c)=(a\rightharpoonup b)\cdot(a\rightharpoonup c). 
\]
Moreover, the following deformed associativity holds:
\[
a\rightharpoonup(b\rightharpoonup c)=(a\cdot(a\rightharpoonup b))\rightharpoonup c.
\]
As shown in \cite{AEM}, the notion of post-group formalises
properties of group-like elements in the completion of the enveloping algebra of a post-Lie algebra, see also \cite{ELM}. The category of post-groups is proven to be isomorphic to the category of skew braces \cite[Theorem 3.25]{BGST}. \medskip

We also recall that, given an action $\varphi:\mathfrak{g}\to\mathrm{Der}(\mathfrak{h})$ of a Lie algebra $(\mathfrak{g},[\cdot,\cdot]_{\mathfrak{g}})$ on a Lie algebra $(\mathfrak{h},[\cdot,\cdot]_{\mathfrak{h}})$, a linear map $R:\mathfrak{h}\to\mathfrak{g}$ is called a \textit{relative Rota–Baxter operator of weight 1} on $\mathfrak{g}$ with respect to $(\mathfrak{h},\varphi)$ if the following condition holds true for all $x,y\in\mathfrak{h}$:
\[
[R(x),R(y)]_{\mathfrak{g}}=R(\varphi(R(x))y-\varphi(R(y))x+[x,y]_{\mathfrak{h}}).
\]
It was shown in \cite{BGN} that, given a relative Rota–Baxter operator $R:\mathfrak{h}\to\mathfrak{g}$, 
the Lie algebra $\mathfrak{h}$ inherits a post-Lie algebra structure defined by
\[
x\rightharpoonup y:=\varphi(R(x))y,\quad \text{for all}\ x,y\in\mathfrak{h}
\]
and, given a post-Lie algebra,  the identity map is a relative Rota–Baxter operator on the subadjacent Lie algebra. Similarly, given an action $\phi:G\to\mathrm{Aut}(H)$ of a group $G$ on a group $H$, a map $R:H\to G$ is called a \textit{relative Rota–Baxter operator of weight 1} on $G$ with respect to $(H,\phi)$ if
\[
R(h)\cdot_{G}R(k)=R(h\cdot_{H}\phi(R(h))k),\quad \text{for all}\ h,k\in H.
\]
This notion was introduced in \cite{GLS} and its connection with skew braces was studied in \cite{BG}. \medskip

\noindent\textbf{Yetter--Drinfeld braces and matched pairs}. The notion of \textit{skew brace} was introduced in \cite{GV}, generalising that of \textit{brace} given in \cite{Ru}. The Hopf-theoretic version of a skew brace, namely a \textit{Hopf brace}, was introduced in \cite{AGV}. Hopf braces are equivalent to bijective 1-cocycles, see \cite[Theorem 1.12]{AGV}. Moreover, in the cocommutative case, Hopf braces are also equivalent to matched pairs of Hopf algebras in the sense of \cite[Definition 7.2.1]{Majid2} which satisfy an additional relation (precisely \eqref{braided-commutativity} below), see \cite[Theorem 3.3]{AGV}. In \cite{FS} we have introduced Yetter--Drinfeld braces, which coincide with Hopf braces in the cocommutative case, and in \cite[Theorem 3.25]{FS} we proved that the category of Yetter--Drinfeld braces is isomorphic to the category of matched pairs of actions on a Hopf algebra. 

Let us recall the notions of Yetter--Drinfeld brace and matched pair of actions on a Hopf algebra. In order to do this we first recall the definition of Yetter--Drinfeld module:

\begin{definition}
    Let $(H,\bullet,1,\Delta,\epsilon,T)$ be a Hopf algebra. A \textit{(left-left) Yetter--Drinfeld module} on $H$ is the datum of a left $H$-module $V$ with an action $\rightharpoonup:H\ot V\to V$, which is also a left $H$-comodule with a coaction $\rho:V\to H\ot V$, satisfying the following compatibility condition:
\[
\rho(h\rightharpoonup v)=h_{1}\bullet v_{-1}\bullet T(h_{3})\ot(h_{2}\rightharpoonup v_{0}).
\]
A morphism of Yetter--Drinfeld modules is a morphism of both left $H$-modules and left $H$-comodules. We denote by $^{H}_{H}\mathcal{YD}$ the category of Yetter–Drinfeld modules over $H$.
\end{definition}

More generally, one can define Yetter--Drinfeld modules over a bialgebra but we will work with Hopf algebras in the following. \medskip

It is known that, see e.g. \cite{Rad}, given any monoidal category $(\mathcal{M},\otimes,\mathbf{1})$, one can consider the categories of algebras $\mathrm{Mon}(\Mm)$ and coalgebras $\mathrm{Comon}(\Mm)$ in $\mathcal{M}$. Once the category is braided (or even just prebraided) one can also define the categories of bialgebras $\mathrm{Bimon}(\Mm)$ and Hopf algebras $\mathrm{Hopf}(\Mm)$ in $\Mm$.

The category $^{H}_{H}\mathcal{YD}$ of Yetter--Drinfeld modules over a Hopf algebra $H$ is monoidal and it is equipped with a prebraiding defined, for all $X,Y\in{}^{H}_{H}\mathcal{YD}$, as 
\[
c_{X,Y}:X\ot Y\to Y\ot X,\ x\ot y\mapsto (x_{-1}\rightharpoonup y)\ot x_{0} 
\]
and this is bijective if the antipode of $H$ is bijective, see e.g. \cite{Rad}. Thus, one can define Hopf algebras $\mathrm{Hopf}(^{H}_{H}\mathcal{YD})$ in the category $^{H}_{H}\mathcal{YD}$.

\begin{definition}[{\cite[Definition 3.16]{FS}}]
    A \textit{Yetter–Drinfeld brace} (or YD-brace) $(H,\cdot,\bullet,1,\Delta,\epsilon,S,T)$ is the datum of a Hopf algebra $H^{\bullet}=(H,\bullet,1,\Delta,\epsilon,T)$, a second operation $\cdot$ on $H$, and a linear map $S:H\to H$ such that:
\begin{itemize}
    \item[1)] $(H,\cdot,1,\Delta,\epsilon,S)$ is in $\mathrm{Hopf}(^{H^{\bullet}}_{H^{\bullet}}\mathcal{YD})$ with the action $\rightharpoonup$ defined by $a\rightharpoonup b:=S(a_{1})\cdot(a_{2}\bullet b)$, and the coaction given by $\mathrm{Ad}_{L}$;\medskip
    \item[2)] if we define $a\leftharpoonup b:=T(a_{1}\rightharpoonup b_{1})\bullet a_{2}\bullet b_{2}$, the two maps $\rightharpoonup,\leftharpoonup$ satisfy \eqref{MP5};\medskip
    \item[3)] the two operations $\bullet$ and $\cdot$ satisfy the Hopf brace compatibility:
\begin{equation}\label{HBC}
    a\bullet(b\cdot c)=(a_{1}\bullet b)\cdot S(a_{2})\cdot(a_{3}\bullet c).
\end{equation}
\end{itemize}
Given two Yetter--Drinfeld braces $H$ and $K$, a \emph{morphism of Yetter--Drinfeld braces} is a map $f\colon H\to K$ that is a morphism in $\mathrm{Hopf}(\mathsf{Vec}_\Bbbk)$ between the respective Hopf algebras in $\mathsf{Vec}_\Bbbk$, and satisfies
\[
f(a\cdot_H b)=f(a)\cdot_K f(b).
\]             
In particular, this implies $S_{K}f=fS_{H}$, $f\rightharpoonup_{H}\;=\;\rightharpoonup_{K}\!(f\otimes f)$ and $(\mathrm{Ad}_{L})_Kf=(f\ot f)(\mathrm{Ad}_{L})_H$. We denote the category of Yetter--Drinfeld braces by $\mathcal{YD}\mathrm{Br}(\mathrm{Vec}_\Bbbk)$.
\end{definition}

\begin{definition}[cf. {\cite[Definition 2.1]{FS}}]\label{def:MPA}
    Let $H$ be a bialgebra. A \emph{matched pair of actions} $(H,\rightharpoonup,\leftharpoonup)$ on $H$ is the datum of a left action $\rightharpoonup$ and a right action $\leftharpoonup$ of $H$ on itself, such that $H$ is a left $H$-module coalgebra and a right $H$-module coalgebra with the respective actions, and the following conditions hold for all $a,b,c\in H$:
\begin{align}
        \label{matchedpair-1}
        a\rightharpoonup 1 &= \epsilon(a)1
        ,\\
        \label{matchedpair-2}
        1\leftharpoonup a &= \epsilon(a)1
        ,\\
        \label{matchedpair-3}
         a\rightharpoonup(b\cdot c)&=(a_{1}\rightharpoonup b_{1})\cdot ((a_{2}\leftharpoonup b_{2})\rightharpoonup c),\\
        \label{matchedpair-4}
         (a\cdot b)\leftharpoonup c&=(a\leftharpoonup(b_{1}\rightharpoonup c_{1}))\cdot (b_{2}\leftharpoonup c_{2}),\\
        \label{braided-commutativity}  a\cdot b &= (a_1\rightharpoonup b_1)\cdot (a_2\leftharpoonup b_2).
    \end{align}
    A \emph{morphism of matched pairs of actions} between $(H,\rightharpoonup_H, \leftharpoonup_H)$ and $(K,\rightharpoonup_K, \leftharpoonup_K)$ is a morphism of bialgebras $H\to K$ that intertwines the two left actions and the two right actions, respectively, in the following sense:
    \[f(a\rightharpoonup_H b) = f(a)\rightharpoonup_K f(b),\quad f(a\leftharpoonup_H b) = f(a)\leftharpoonup_K f(b). \] The category of matched pairs of actions (in $\mathrm{Vec}_\Bbbk$) is denoted by $\mathrm{MP}(\mathrm{Vec}_\Bbbk)$.
\end{definition}
Given a matched pair of actions $(H,\rightharpoonup,\leftharpoonup)$ on a Hopf algebra $H$, the following equality is always satisfied
\begin{equation}\label{MP5}
    (a_{1}\rightharpoonup b_{1})\otimes(a_{2}\leftharpoonup b_{2})=(a_{2}\rightharpoonup b_{2})\otimes(a_{1}\leftharpoonup b_{1}),
\end{equation}
see \cite[Corollary 3.7]{FS}. Thus, matched pair of actions on a Hopf algebra form a subclass of matched pair of Hopf algebras in the sense of \cite[Definition 7.2.1]{Majid2}. \medskip



We also recall that in \cite[Definition 4.1]{FS} we introduced the category of Yetter--Drinfeld 1-cocycles and in \cite[Theorem 4.3]{FS} we proved that a subcategory of it is equivalent to the category of Yetter--Drinfeld braces. We will give more details about Yetter--Drinfeld 1-cocycles in the last section. 


\section{Yetter--Drinfeld post-Hopf algebras}\label{sec:YetterDrinfeldpostHopf}

\noindent The notion of \textit{post-Hopf algebra} is introduced in \cite[Definition 2.1]{YYT} and in \cite[Theorem 2.7]{YYT} it is shown that the subspace of primitive elements of a post-Hopf algebra is a post-Lie algebra. Moreover, in \cite[Theorem 2.13]{YYT2} (where \cite{YYT2} is the arXiv version of \cite{YYT}) it was also proven that, in the cocommutative case, one can obtain a Hopf brace from a post-Hopf algebra and vice versa. Here we introduce a generalisation of the notion of post-Hopf algebra, which we call \textit{Yetter--Drinfeld post-Hopf algebra}, that coincides with a post-Hopf algebra in the cocommutative case. The category of Yetter--Drinfeld post-Hopf algebras will turn out to be isomorphic to the category of Yetter--Drinfeld braces given in \cite{FS}.

\begin{definition}
    A \textbf{Yetter--Drinfeld post-Hopf algebra} $(H,\cdot,1,\Delta,\epsilon,S,\rightharpoonup)$, also denoted as a pair $(H,\rightharpoonup)$, is the datum of an algebra $(H,\cdot,1)$, a coalgebra $(H,\Delta,\epsilon)$, a morphism $S:H\to H$ such that $x_{1}\cdot S(x_{2})=S(x_{1})\cdot x_{2}=\epsilon(x)1_{H}$ for all $x\in H$ and a morphism of coalgebras $\rightharpoonup:H\ot H\to H$ satisfying the following equalities, for any $x,y,z\in H$:
\begin{equation}\label{dotlinear}
x\rightharpoonup(y\cdot z)=(x_{1}\rightharpoonup y)\cdot(x_{2}\rightharpoonup z),
\end{equation}    
\begin{equation}\label{deformedassociativityright}
x\rightharpoonup(y\rightharpoonup z)=\big(x_{1}\cdot(x_{2}\rightharpoonup y)\big)\rightharpoonup z.
\end{equation} 
Moreover, the morphism $\alpha_{\rightharpoonup}:H\to\mathrm{End}(H)$ defined by
\[
(\alpha_{\rightharpoonup}x)(y):=x\rightharpoonup y,\ \text{for all}\ x,y\in H,
\]
is convolution invertible in $\mathrm{Hom}(H,\mathrm{End}(H))$, i.e. there exists a unique morphism $\beta_{\rightharpoonup}:H\to\mathrm{End}(H)$ such that 
\begin{equation}\label{alphainvconv}
(\alpha_{\rightharpoonup}x_{1})\circ(\beta_{\rightharpoonup}x_{2})=(\beta_{\rightharpoonup}x_{1})\circ(\alpha_{\rightharpoonup}x_{2})=\epsilon(x)\mathrm{Id}_{H},\quad \text{for all}\ x\in H.
\end{equation}
Moreover, $\epsilon$ is a morphism of algebras, i.e. $\epsilon(a\cdot b)=\epsilon(a)\epsilon(b)$ and $\epsilon(1_{H})=1_{\Bbbk}$, $\Delta(1_{H})=1_{H}\otimes1_{H}$ and the following compatibility condition between $\Delta$ and $\cdot$ holds true:
\begin{equation}\label{comp.Deltadot}
\Delta(x\cdot y)
=\Big(x_{1}\cdot\big(\alpha_{\rightharpoonup}x_{2}(\beta_{\rightharpoonup}x_{4}(y_{1}))\big)\Big)\otimes\big(x_{3}\cdot y_{2}\big).
\end{equation}
Finally, defined 
\begin{equation}\label{defbullet}
x\bullet_{\rightharpoonup}y:=x_{1}\cdot(x_{2}\rightharpoonup y),
\end{equation}
\begin{equation}\label{defantipode}
    S_{\rightharpoonup}(x):=\beta_{\rightharpoonup}x_{1}(S(x_{2})),
\end{equation}
and 
\begin{equation}\label{defleftharp}
x\leftharpoonup y:=S_{\rightharpoonup}(x_{1}\rightharpoonup y_{1})\bullet_{\rightharpoonup}x_{2}\bullet_{\rightharpoonup}y_{2}, 
\end{equation}
we require that $S_{\rightharpoonup}$ is anti-comultiplicative, i.e. it satisfies $\Delta S_{\rightharpoonup}(x)=S_{\rightharpoonup}(x_{2})\otimes S_{\rightharpoonup}(x_{1})$, and the pair $(\rightharpoonup,\leftharpoonup)$ satisfies \eqref{MP5}.

A \textit{morphism of Yetter--Drinfeld post-Hopf algebras} from $(H,\rightharpoonup)$ to $(H',\rightharpoonup')$ is a morphism of algebras and coalgebras $g:H\to H'$ satisfying
\begin{equation}\label{morphismYDpost}
g(x\rightharpoonup y)=g(x)\rightharpoonup'g(y),\quad \text{for all}\ x,y\in H.
\end{equation}
We denote the category of Yetter--Drinfeld post-Hopf algebras and morphisms of Yetter--Drinfeld post-Hopf algebras by $\mathcal{YD}\mathrm{PH}(\mathrm{Vec}_{\Bbbk})$. 
\end{definition}

\begin{remark}
    We use notations $\bullet_{\rightharpoonup}$ and $S_{\rightharpoonup}$, similarly to those used in \cite[Theorem 2.5]{YYT}. Let us observe that one can recover $\rightharpoonup$ as 
\begin{equation}\label{harpfromS}
    x\rightharpoonup y=S(x_{1})\cdot x_{2}\cdot(x_{3}\rightharpoonup y)\overset{\eqref{defbullet}}{=}S(x_{1})\cdot(x_{2}\bullet_{\rightharpoonup}y)
\end{equation}
and rewrite \eqref{deformedassociativityright} as
\begin{equation}\label{associativity}
x\rightharpoonup(y\rightharpoonup z)=
(x\bullet_{\rightharpoonup}y)\rightharpoonup z.
\end{equation}
\end{remark}
\begin{remark}
Suppose $H$ is cocommutative. Then, we obtain 
\[
\begin{split}
\Delta(x\cdot y)&\overset{\eqref{comp.Deltadot}}{=}\Big(x_{1}\cdot\big(\alpha_{\rightharpoonup}x_{2}(\beta_{\rightharpoonup}x_{4}(y_{1}))\big)\Big)\otimes\big(x_{3}\cdot y_{2}\big)=\Big(x_{1}\cdot\big(\alpha_{\rightharpoonup}x_{2}(\beta_{\rightharpoonup}x_{3}(y_{1}))\big)\Big)\otimes\big(x_{4}\cdot y_{2}\big)\\&\overset{\eqref{alphainvconv}}{=}(x_{1}\cdot y_{1})\otimes(x_{2}\cdot y_{2}),
\end{split},
\]
hence $H$ becomes a standard Hopf algebra and $(H,\rightharpoonup)$ is a post-Hopf algebra, see \cite[Definition 2.1]{YYT}. Moreover, clearly, \eqref{MP5} is automatically satisfied and, as we will observe later, $S_{\rightharpoonup}$ becomes comultiplicative.
\end{remark}

One can prove that the same properties given in \cite[Lemma 2.4]{YYT} hold for Yetter--Drinfeld post-Hopf algebras. We include the proof for the sake of completeness, showing that there are no
problems with this more general definition.

\begin{lemma}
    Let $(H,\rightharpoonup)$ be a Yetter--Drinfeld post-Hopf algebra. Then, the following properties are satisfied, for all $x,y\in H$:
\begin{equation}\label{ulinear}
    x\rightharpoonup1=\epsilon(x)1,
\end{equation}
\begin{equation}\label{harpoonunit}
    1\rightharpoonup x=x,
\end{equation}  
\begin{equation}\label{Slinear}
S(x\rightharpoonup y)=x\rightharpoonup S(y).
\end{equation} 
\end{lemma}

\begin{proof}
    Since $\rightharpoonup$ is a morphism of coalgebras, we have
\[
x\rightharpoonup1=(x_{1}\rightharpoonup1)\epsilon(x_{2}\rightharpoonup1)=(x_{1}\rightharpoonup1)\cdot(x_{2}\rightharpoonup1)\cdot S(x_{3}\rightharpoonup1)\overset{\eqref{dotlinear}}{=}(x_{1}\rightharpoonup1)\cdot S(x_{2}\rightharpoonup1)=\epsilon(x\rightharpoonup1)1=\epsilon(x)1.
\]
By \eqref{alphainvconv} we have $(\alpha_{\rightharpoonup}1)\circ(\beta_{\rightharpoonup}1)=(\beta_{\rightharpoonup}1)\circ(\alpha_{\rightharpoonup}1)=\mathrm{Id}_{H}$, so $\alpha_{\rightharpoonup}1$ is an automorphism of $H$ with inverse given by $\beta_{\rightharpoonup}1$. Moreover, we have
\[
(\alpha_{\rightharpoonup}1)^{2}(x)=1\rightharpoonup(1\rightharpoonup x)\overset{\eqref{deformedassociativityright}}{=}(1\rightharpoonup1)\rightharpoonup x\overset{\eqref{ulinear}}{=}1\rightharpoonup x=(\alpha_{\rightharpoonup}1)(x),
\]
hence $1\rightharpoonup x=(\alpha_{\rightharpoonup}1)(x)=x$ (and so also $\beta_{\rightharpoonup}1=\mathrm{Id}$). Finally, we have
\[
\begin{split}
    S(x\rightharpoonup y)&=S(x_{1}\rightharpoonup y_{1})\epsilon(x_{2})\epsilon(y_{2})\overset{\eqref{ulinear}}{=}S(x_{1}\rightharpoonup y_{1})\cdot(x_{2}\rightharpoonup\epsilon(y_{2})1)\\&=S(x_{1}\rightharpoonup y_{1})\cdot\big(x_{2}\rightharpoonup y_{2}\cdot S(y_{3})\big)\overset{\eqref{dotlinear}}{=}S(x_{1}\rightharpoonup y_{1})\cdot(x_{2}\rightharpoonup y_{2})\cdot(x_{3}\rightharpoonup S(y_{3}))\\&=\epsilon(x_{1}\rightharpoonup y_{1})(x_{2}\rightharpoonup S(y_{2}))=
    x\rightharpoonup S(y).
\end{split}
\]
\end{proof}

\begin{remark}
Given a Yetter--Drinfeld post-Hopf algebra $(H,\rightharpoonup)$ we know by definition that $\alpha_{\rightharpoonup}$ has convolution inverse given by $\beta_{\rightharpoonup}$. We prove that $\alpha_{\rightharpoonup}S_{\rightharpoonup}$ is a right convolution inverse of $\alpha_{\rightharpoonup}$ in order to obtain $\beta_{\rightharpoonup}=\alpha_{\rightharpoonup}S_{\rightharpoonup}$. We compute
\[
\begin{split}
\alpha_{\rightharpoonup}x_{1}(\alpha_{\rightharpoonup}S_{\rightharpoonup}(x_{2})(a))&=
x_{1}\rightharpoonup(S_{\rightharpoonup}(x_{2})\rightharpoonup a)\overset{\eqref{deformedassociativityright}}{=}(x_{1}\cdot(x_{2}\rightharpoonup S_{\rightharpoonup}(x_{3})))\rightharpoonup a\\&\overset{\eqref{defantipode}}{=}\Big(x_{1}\cdot\big(\alpha_{\rightharpoonup}x_{2}(\beta_{\rightharpoonup}x_{3}(S(x_{4})))\big)\Big)\rightharpoonup a\overset{\eqref{alphainvconv}}{=}(x_{1}\cdot S(x_{2}))\rightharpoonup a\overset{\eqref{harpoonunit}}{=}\epsilon(x)a,
\end{split}
\]
thus we get
\begin{equation}\label{betarightharp}
\alpha_{\rightharpoonup}S_{\rightharpoonup}=\beta_{\rightharpoonup}.
\end{equation}
\end{remark}

\begin{remark}\label{Hdotmoncomon}
Notice that, from \eqref{associativity} and \eqref{harpoonunit}, we have that $H$ is an object in $_{H_{\rightharpoonup}}\mm$, the category of left $H_{\rightharpoonup}$-modules. Moreover, from \eqref{dotlinear} and \eqref{ulinear}, we have that $(H,\cdot,1)$ is in $\mathrm{Mon}(_{H_{\rightharpoonup}}\mm)$. Recall also that, by definition of Yetter--Drinfeld post-Hopf algebra, $\rightharpoonup$ is a morphism of coalgebras, i.e. $(H,\Delta,\epsilon)$ is in $\mathrm{Comon}(_{H_{\rightharpoonup}}\mm)$.  
\end{remark}

We can prove the following result.

\begin{lemma}
The following equations are satisfied, for all $x\in H$:
\begin{equation}\label{deltalpha}
\Delta\circ(\alpha_{\rightharpoonup}x)=(\alpha_{\rightharpoonup}x_{1}\otimes\alpha_{\rightharpoonup}x_{2})\circ\Delta,
\end{equation}
\begin{equation}\label{deltabeta}
\Delta\circ(\beta_{\rightharpoonup}x)=(\beta_{\rightharpoonup}x_{2}\otimes\beta_{\rightharpoonup}x_{1})\circ\Delta,
\end{equation}
\begin{equation}\label{dotalpha}
(\alpha_{\rightharpoonup}x)\circ m_{\cdot}=m_{\cdot}\circ(\alpha_{\rightharpoonup}x_{1}\otimes\alpha_{\rightharpoonup}x_{2}),
\end{equation}
\begin{equation}\label{dotbeta}
(\beta_{\rightharpoonup}x)\circ m_{\cdot}=m_{\cdot}\circ(\beta_{\rightharpoonup}x_{2}\otimes\beta_{\rightharpoonup}x_{1}).
\end{equation}
\end{lemma}

\begin{proof}
First we compute
\[
\Delta((\alpha_{\rightharpoonup}x)(y))=\Delta(x\rightharpoonup y)=(x_{1}\rightharpoonup y_{1})\otimes(x_{2}\rightharpoonup y_{2})=(\alpha_{\rightharpoonup}x_{1})(y_{1})\otimes(\alpha_{\rightharpoonup}x_{2})(y_{2})=(\alpha_{\rightharpoonup}x_{1}\otimes\alpha_{\rightharpoonup}x_{2})\Delta(y),
\]
hence \eqref{deltalpha} holds true. Thus, we have
\begin{equation}\label{teccond}
\begin{split}
(\alpha_{\rightharpoonup}x_{1}\otimes\alpha_{\rightharpoonup}x_{3})(\beta_{\rightharpoonup}x_{2}\otimes\beta_{\rightharpoonup}x_{4})\Delta(y)&\overset{\eqref{alphainvconv}}{=}\epsilon(x_{1})\epsilon(x_{2})\Delta(y)=\Delta(\epsilon(x)y)\overset{\eqref{alphainvconv}}{=}\Delta(\alpha_{\rightharpoonup}x_{1}(\beta_{\rightharpoonup}x_{2}(y)))\\&\overset{\eqref{deltalpha}}{=}(\alpha_{\rightharpoonup}x_{1}\otimes\alpha_{\rightharpoonup}x_{2})\Delta(\beta_{\rightharpoonup}x_{3}(y)),
\end{split}
\end{equation}
hence 
\begin{equation}\label{teccond2}
\begin{split}
(\mathrm{Id}\otimes\alpha_{\rightharpoonup}x_{1})\Delta(\beta_{\rightharpoonup}x_{2}(y))&\overset{\eqref{alphainvconv}}{=}(\beta_{\rightharpoonup}x_{1}\alpha_{\rightharpoonup}x_{2}\otimes\alpha_{\rightharpoonup}x_{3})\Delta(\beta_{\rightharpoonup}x_{4}(y))\\&\overset{\eqref{teccond}}{=}(\beta_{\rightharpoonup}x_{1}\alpha_{\rightharpoonup}x_{2}\otimes\alpha_{\rightharpoonup}x_{4})(\beta_{\rightharpoonup}x_{3}\otimes\beta_{\rightharpoonup}x_{5})\Delta(y)\\&\overset{\eqref{alphainvconv}}{=}(\mathrm{Id}\otimes\alpha_{\rightharpoonup}x_{2})(\beta_{\rightharpoonup}x_{1}\otimes\beta_{\rightharpoonup}x_{3})\Delta(y)
\end{split}
\end{equation}
and then we obtain
\[
\begin{split}
\Delta((\beta_{\rightharpoonup}x)(y))&\overset{\eqref{alphainvconv}}{=}(\mathrm{Id}\otimes\beta_{\rightharpoonup}x_{1}\alpha_{\rightharpoonup}x_{2})\Delta(\beta_{\rightharpoonup}x_{3}(y))\\&\overset{\eqref{teccond2}}{=}(\mathrm{Id}\otimes\beta_{\rightharpoonup}x_{1}\alpha_{\rightharpoonup}x_{3})(\beta_{\rightharpoonup}x_{2}\otimes\beta_{\rightharpoonup}x_{4})\Delta(y)\\&\overset{\eqref{alphainvconv}}{=}(\beta_{\rightharpoonup}x_{2}\otimes\beta_{\rightharpoonup}x_{1})\Delta(y).
\end{split}
\]
Hence \eqref{deltabeta} is satisfied. Moreover, we have
\[
(\alpha_{\rightharpoonup}x)m_{\cdot}(a\otimes b)=x\rightharpoonup(a\cdot b)\overset{\eqref{dotlinear}}{=}(x_{1}\rightharpoonup a)\cdot(x_{2}\rightharpoonup b)=m_{\cdot}(\alpha_{\rightharpoonup}x_{1}\otimes\alpha_{\rightharpoonup}x_{2})(a\otimes b),
\]
so \eqref{dotalpha} holds true. The proof of \eqref{dotbeta} is the dual of the proof of \eqref{deltabeta}.
\end{proof}

\begin{remark}
Notice that 
\[
\Delta(S_{\rightharpoonup}(x))\overset{\eqref{defantipode}}{=}\Delta\beta_{\rightharpoonup}x_{1}(S(x_{2}))\overset{\eqref{deltabeta}}{=}(\beta_{\rightharpoonup}x_{2}\otimes\beta_{\rightharpoonup}x_{1})\Delta(S(x_{3}))=\beta_{\rightharpoonup}x_{2}(S(x_{4}))\otimes\beta_{\rightharpoonup}x_{1}(S(x_{3})),
\]
hence the fact that $S_{\rightharpoonup}$ is anti-comultiplicative reads in terms on $\beta_{\rightharpoonup}$ as:
\[
\beta_{\rightharpoonup}x_{2}(S(x_{4}))\otimes\beta_{\rightharpoonup}x_{1}(S(x_{3}))=\beta_{\rightharpoonup}x_{3}(S(x_{4}))\otimes\beta_{\rightharpoonup}x_{1}(S(x_{2})).
\]
\end{remark}

\begin{remark}
    Notice that
\[
\beta_{\rightharpoonup}x(1)=\beta_{\rightharpoonup}x_{1}(\epsilon(x_{2})1)\overset{\eqref{ulinear}}{=}\beta_{\rightharpoonup}x_{1}(\alpha_{\rightharpoonup}x_{2}(1))\overset{\eqref{alphainvconv}}{=}\epsilon(x)1
\]
and then
\begin{equation}\label{propertyantipode2}
    \beta_{\rightharpoonup}x_{2}(S(x_{3}))\cdot\beta_{\rightharpoonup}x_{1}(x_{4})\overset{\eqref{dotbeta}}{=}\beta_{\rightharpoonup}x_{1}(S(x_{2})\cdot x_{3})=\beta_{\rightharpoonup}x(1)=\epsilon(x)1,
\end{equation}
for all $x\in H$. Observe also that if $H$ is cocommutative we have $\Delta\circ(\beta_{\rightharpoonup}x)=(\beta_{\rightharpoonup}x_{1}\otimes\beta_{\rightharpoonup}x_{2})\Delta$ and $(\beta_{\rightharpoonup}x)\circ m_{\cdot}=m_{\cdot}\circ(\beta_{\rightharpoonup}x_{1}\otimes\beta_{\rightharpoonup}x_{2})$ for all $x\in H$ and moreover $S_{\rightharpoonup}$ is comultiplicative.
\end{remark}

The following result generalises \cite[Theorem 2.5]{YYT}, removing the hypothesis of cocommutativity.

\begin{proposition}\label{HharpHopf}
    Let $(H,\rightharpoonup)$ be a Yetter--Drinfeld post-Hopf algebra and define $\bullet_{\rightharpoonup}$ and $S_{\rightharpoonup}$ as in \eqref{defbullet} and \eqref{defantipode}, respectively. 
Then, $H_{\rightharpoonup}:=(H,\bullet_{\rightharpoonup},1,\Delta,\epsilon,S_{\rightharpoonup})$ is a Hopf algebra. We call $H_{\rightharpoonup}$ the subadjacent Hopf algebra in analogy with \cite{YYT}.
\end{proposition}

\begin{proof}
By definition of Yetter--Drinfeld post-Hopf algebra we already know that $\Delta(1_{H})=1_{H}\otimes1_{H}$ and $\epsilon(1_{H})=1_{\Bbbk}$. We show that $\Delta$ and $\epsilon$ are multiplicative with respect to $\bullet_{\rightharpoonup}$: 
\[
\begin{split}
\Delta(x\bullet_{\rightharpoonup}y)&\overset{\eqref{defbullet}}{=}\Delta(x_{1}\cdot(x_{2}\rightharpoonup y))\overset{\eqref{comp.Deltadot}}{=}\Bigg(x_{1}\cdot\Big(\alpha_{\rightharpoonup}x_{2}\big(\beta_{\rightharpoonup}x_{4}(x_{5}\rightharpoonup y_{1})\big)\Big)\Bigg)\otimes\big(x_{3}\cdot(x_{6}\rightharpoonup y_{2})\big)\\&\ =x_{1}\cdot\Big(\alpha_{\rightharpoonup}x_{2}\big(\beta_{\rightharpoonup}x_{4}(\alpha_{\rightharpoonup}x_{5}(y_{1}))\big)\Big)\otimes(x_{3}\cdot(x_{6}\rightharpoonup y_{2}))\\&\overset{\eqref{alphainvconv}}{=}(x_{1}\cdot(x_{2}\rightharpoonup y_{1}))\otimes(x_{3}\cdot(x_{4}\rightharpoonup  y_{2}))\\&\ =(x_{1}\bullet_{\rightharpoonup}y_{1})\otimes(x_{2}\bullet_{\rightharpoonup}y_{2})
\end{split}
\]
and
\[
\epsilon(x\bullet_{\rightharpoonup}y)\overset{\eqref{defbullet}}{=}\epsilon(x_{1}\cdot(x_{2}\rightharpoonup y))=\epsilon(x_{1})\epsilon(x_{2}\rightharpoonup y)=\epsilon(x_{1})\epsilon(x_{2})\epsilon(y)=\epsilon(x)\epsilon(y).
\]
We show that $(H,\bullet_{\rightharpoonup},1)$ is an algebra, so that $(H,\bullet_{\rightharpoonup},1,\Delta,\epsilon)$ is a bialgebra. We have
\[
1\bullet_{\rightharpoonup}x=1\cdot(1\rightharpoonup x)\overset{\eqref{harpoonunit}}{=}1\cdot x=x,\qquad x\bullet_{\rightharpoonup}1=x_{1}\cdot(x_{2}\rightharpoonup1)\overset{\eqref{ulinear}}{=}x_{1}\cdot\epsilon(x_{2})1=x
\]
and also
\[
\begin{split}
x\bullet_{\rightharpoonup}(y\bullet_{\rightharpoonup}z&)=x_{1}\cdot(x_{2}\rightharpoonup(y_{1}\cdot(y_{2}\rightharpoonup z)))\\&\overset{\eqref{dotlinear}}{=}x_{1}\cdot(x_{2}\rightharpoonup y_{1})\cdot(x_{3}\rightharpoonup(y_{2}\rightharpoonup z))\\&\overset{\eqref{deformedassociativityright}}{=}(x_{1}\cdot(x_{2}\rightharpoonup y_{1}))\cdot\big((x_{3}\cdot(x_{4}\rightharpoonup y_{2}))\rightharpoonup z\big)\\&\ =(x_{1}\bullet_{\rightharpoonup} y_{1})\cdot((x_{2}\bullet_{\rightharpoonup} y_{2})\rightharpoonup z)\\&\ =(x\bullet_{\rightharpoonup}y)\bullet_{\rightharpoonup}z,
\end{split}
\]
where the last equality follows from the fact that $\Delta$ is multiplicative with respect to $\bullet_{\rightharpoonup}$. Finally, we prove that $S_{\rightharpoonup}$ is an antipode. We compute
\[
\begin{split}
    x_{1}\bullet_{\rightharpoonup}S_{\rightharpoonup}(x_{2})&\overset{\eqref{defbullet}}{=}x_{1}\cdot(x_{2}\rightharpoonup S_{\rightharpoonup}(x_{3}))\overset{\eqref{defantipode}}{=}x_{1}\cdot\Big(\alpha_{\rightharpoonup}x_{2}\big(\beta_{\rightharpoonup}x_{3}(S(x_{4}))\big)\Big)=x_{1}\cdot S(x_{2})=\epsilon(x)1.
\end{split}
\]
Moreover, 
since $S_{\rightharpoonup}$ is anti-comultiplicative, we obtain 
\[
\begin{split}
S_{\rightharpoonup}(x_{1})\bullet_{\rightharpoonup}x_{2}&\overset{\eqref{defbullet}}{=}S_{\rightharpoonup}(x_{1})_{1}\cdot(S_{\rightharpoonup}(x_{1})_{2}\rightharpoonup x_{2})\\&\ =S_{\rightharpoonup}(x_{2})\cdot(S_{\rightharpoonup}(x_{1})\rightharpoonup x_{3})\\&\ =\beta_{\rightharpoonup}x_{2}(S(x_{3}))\cdot(\alpha_{\rightharpoonup}S_{\rightharpoonup}x_{1}(x_{4}))\\&\overset{\eqref{betarightharp}}{=}\beta_{\rightharpoonup}x_{2}(S(x_{3}))\cdot\beta_{\rightharpoonup}x_{1}(x_{4})\\&\overset{\eqref{propertyantipode2}}{=}\epsilon(x)1,
\end{split}
\]
hence $S_{\rightharpoonup}$ is an antipode and $(H,\bullet_{\rightharpoonup},1,\Delta,\epsilon,S_{\rightharpoonup})$ is a Hopf algebra. 
\end{proof}

\begin{remark}
Let us observe that, since $S_{\rightharpoonup}$ is an antipode, one can recover $\cdot$ as
\begin{equation}\label{dotfrombullet}
a_{1}\bullet_{\rightharpoonup}(S_{\rightharpoonup}(a_{2})\rightharpoonup b)=a_{1}\cdot(a_{2}\rightharpoonup (S_{\rightharpoonup}(a_{3})\rightharpoonup b))=a_{1}\cdot(a_{2}\bullet_{\rightharpoonup}S_{\rightharpoonup}(a_{3})\rightharpoonup b)=a\cdot b.
\end{equation}
\end{remark}

The relation between post-Hopf algebras and post-Lie algebras given in \cite[Theorem 2.7]{YYT} still remains true for Yetter--Drinfeld post-Hopf algebras as it is written in the following result. 

\begin{proposition}
     Let $(H,\rightharpoonup)$ be a Yetter--Drinfeld post-Hopf algebra. Then, its subspace of primitive elements $P(H)$ is a post-Lie algebra.
\end{proposition}

The proof is exactly the same of \cite[Theorem 2.7]{YYT}. Indeed, this only involves \eqref{dotlinear},\eqref{deformedassociativityright},\eqref{ulinear},\eqref{harpoonunit} and the fact that $\rightharpoonup$ is a morphism of coalgebras. This shows that the notion of Yetter--Drinfeld post-Hopf algebra is the natural one in the not-cocommutative setting.

\begin{invisible}
\begin{proof}
    The proof is the same of that given for \cite[Theorem 2.7]{YYT}, which we include for the sake of completeness. Since $\rightharpoonup$ is a morphism of coalgebras, for all $x,y\in P(H)$, we have
\[
\begin{split}
    \Delta(x\rightharpoonup y)&=(x_{1}\rightharpoonup y_{1})\otimes(x_{2}\rightharpoonup y_{2})\\&=(1\rightharpoonup1)\otimes(x\rightharpoonup y)+(1\rightharpoonup y)\otimes(x\rightharpoonup1)+(x\rightharpoonup1)\otimes(1\rightharpoonup y)+(x\rightharpoonup y)\otimes(1\rightharpoonup1)\\&\overset{\eqref{ulinear},\eqref{harpoonunit}}{=}1\otimes(x\rightharpoonup y)+(x\rightharpoonup y)\otimes1.
\end{split}
\]
Thus, we obtain a linear map $\rightharpoonup:P(H)\otimes P(H)\to P(H)$. For all $x,y\in P(H)$, we have
\begin{equation}\label{relationLiealgebra}
x\rightharpoonup(y\cdot z)\overset{\eqref{dotlinear}}{=}(1\rightharpoonup y)\cdot(x\rightharpoonup z)+(x\rightharpoonup y)\cdot(1\rightharpoonup z)\overset{\eqref{harpoonunit}}{=} y\cdot(x\rightharpoonup z)+(x\rightharpoonup y)\cdot z.
\end{equation}
Given the Lie algebra $(P(H),[\cdot,\cdot])$, where $[x,y]:=x\cdot y-y\cdot x$, we obtain
\[
\begin{split}
x\rightharpoonup[y,z]&=\big(x\rightharpoonup(y\cdot z)\big)-\big(x\rightharpoonup(z\cdot y)\big)\\&\overset{\eqref{relationLiealgebra}}{=} \big(y\cdot(x\rightharpoonup z)\big)+\big((x\rightharpoonup y)\cdot z\big)-\big(z\cdot(x\rightharpoonup y)\big)-\big((x\rightharpoonup z)\cdot y\big)\\&=[x\rightharpoonup y,z]+[y,x\rightharpoonup z],
\end{split}
\]
so \eqref{comp.act-brac} holds true. Moreover, we have
\begin{equation}\label{relationLiealgebra2}
x\rightharpoonup(y\rightharpoonup z)\overset{\eqref{deformedassociativityright}}{=}\Big(\big(1\cdot(x\rightharpoonup y)\big)\rightharpoonup z\Big)+\Big(\big(x\cdot(1\rightharpoonup y)\big)\rightharpoonup z\Big)\overset{\eqref{harpoonunit}}{=}\big((x\rightharpoonup y)\rightharpoonup z\big)+\big((x\cdot y)\rightharpoonup z\big),
\end{equation}
so we obtain
\[
\begin{split}
[x,y]\rightharpoonup z&=\big((x\cdot y)\rightharpoonup z\big)-\big((y\cdot x)\rightharpoonup z\big)\\&\overset{\eqref{relationLiealgebra2}}{=}\big(x\rightharpoonup(y\rightharpoonup z)\big)-\big((x\rightharpoonup y)\rightharpoonup z\big)-\big(y\rightharpoonup(x\rightharpoonup z)\big)+\big((y\rightharpoonup x)\rightharpoonup z\big),
\end{split}
\]
hence \eqref{postLiecondition2} is satisfied. Therefore, $(P(H),[\cdot,\cdot],\rightharpoonup)$ is a post-Lie algebra.
\end{proof}
\end{invisible}

\begin{remark}
    A kind of converse of the previous result holds. Indeed, it is known that, given a post-Lie algebra $(\mathfrak{g},[\cdot,\cdot],\rightharpoonup)$, the map $\rightharpoonup$ can be extended to its universal enveloping algebra and induces a subadjacent Hopf algebra structure isomorphic to the universal enveloping algebra of the subadjacent Lie algebra $\mathfrak{g}_{\rightharpoonup}$, see \cite[Proposition 3.1 and Theorem 3.4]{ELM}. In this case, since universal enveloping algebras are cocommutative, the definition of Yetter--Drinfeld post-Hopf algebra coincides with that of post-Hopf algebra and one can refer to \cite[Theorem 2.8]{YYT}. Observe also that one can easily obtain examples of cocommutative post-Hopf algebras which do not come from Lie algebras considering group algebras with action coming from Rota--Baxter operators on the base group, as it is done in \cite[Example 2.11]{YYT}.
\end{remark}

\begin{remark}
    In \cite[Definition 2.3]{YYT} a post-Hopf algebra $(H,\rightharpoonup)$ is called a \textit{pre-Hopf algebra} if $H$ is commutative. In this case $P(H)$ results to be a pre-Lie algebra since $[x,y]=0$ for all $x,y\in P(H)$. We will introduce the notion of \textit{Yetter--Drinfeld pre-Hopf algebra} later and it will coincide with the notion of pre-Hopf algebra in the cocommutative case.
\end{remark}

\subsection{Yetter--Drinfeld post-Hopf algebras and Yetter--Drinfeld braces}

We are ready to prove that the category of Yetter--Drinfeld post-Hopf algebras is isomorphic to the category of Yetter--Drinfeld braces. First we show some preliminary results.

\begin{lemma}
    Let $(H,\rightharpoonup)$ be a Yetter--Drinfeld post-Hopf algebra and define $\leftharpoonup$ as in \eqref{defleftharp}. Then, we have $\epsilon(a\leftharpoonup b)=\epsilon(a)\epsilon(b)$ and $(a_{1}\rightharpoonup b_{1})\bullet_{\rightharpoonup}(a_{2}\leftharpoonup b_{2})=a\bullet_{\rightharpoonup}b$. 
\end{lemma}

\begin{proof}
We compute
\[
\epsilon(a\leftharpoonup b)=\epsilon(S_{\rightharpoonup}(a_{1}\rightharpoonup b_{1})\bullet_{\rightharpoonup}a_{2}\bullet_{\rightharpoonup}b_{2})=\epsilon(S_{\rightharpoonup}(a_{1}\rightharpoonup b_{1}))\epsilon(a_{2})\epsilon(b_{2})=\epsilon(a\rightharpoonup b)=\epsilon(a)\epsilon(b)
\]
and also
\[
(a_{1}\rightharpoonup b_{1})\bullet_{\rightharpoonup}(a_{2}\leftharpoonup b_{2})=(a_{1}\rightharpoonup b_{1})\bullet_{\rightharpoonup} S_{\rightharpoonup}(a_{2}\rightharpoonup b_{2})\bullet_{\rightharpoonup}a_{3}\bullet_{\rightharpoonup}b_{3}=\epsilon(a_{1}\rightharpoonup b_{1})a_{2}\bullet_{\rightharpoonup}b_{2}=a\bullet_{\rightharpoonup}b.
\]
\end{proof}

\begin{remark}
By \cite[Lemma 3.6]{FS} we have that \eqref{MP5} is equivalent to $\leftharpoonup$ being a morphism of coalgebras, i.e. $\Delta(a\leftharpoonup b)=(a_{1}\leftharpoonup b_{1})\otimes(a_{2}\leftharpoonup b_{2})$
and also to 
\begin{equation}\label{equivalentcondition}
    (a_{1}\rightharpoonup b_{1})\bullet_{\rightharpoonup}(a_{3}\leftharpoonup b_{3})\otimes(a_{2}\rightharpoonup b_{2})=(a_{1}\bullet_{\rightharpoonup} b_{1})\otimes(a_{2}\rightharpoonup b_{2}).
\end{equation}
Let us also observe that, from the definition of $\leftharpoonup$ and the fact that $S_{\rightharpoonup}$ is the antipode of $H_{\rightharpoonup}$, we obtain 
\begin{equation}\label{Sharponharp}
S_{\rightharpoonup}(a\rightharpoonup b)=(a_{1}\leftharpoonup b_{1})\bullet_{\rightharpoonup}S_{\rightharpoonup}(b_{2})\bullet_{\rightharpoonup}S_{\rightharpoonup}(a_{2}). 
\end{equation}
\end{remark}
Now, we show that a Yetter--Drinfeld post-Hopf algebra is automatically a Hopf monoid in the category of (left-left) Yetter--Drinfeld modules over $H_{\rightharpoonup}$.

\begin{theorem}\label{HdotHopfYD}
Given a Yetter--Drinfeld post-Hopf algebra $(H,\rightharpoonup)$ one has that $(H,\rightharpoonup,\mathrm{Ad}_{L})$ is in $^{H_{\rightharpoonup}}_{H_{\rightharpoonup}}\mathcal{YD}$. Moreover, $(H,\cdot,1,\Delta,\epsilon,S)$ is in $\mathrm{Hopf}(^{H_{\rightharpoonup}}_{H_{\rightharpoonup}}\mathcal{YD})$.
\end{theorem}

\begin{proof}
Given the left adjoint coaction $\mathrm{Ad}_{L}:H\mapsto H_{\rightharpoonup}\otimes H$, $a\mapsto a_{1}\bullet_{\rightharpoonup}S_{\rightharpoonup}(a_{3})\otimes a_{2}$ and the action $\rightharpoonup:H_{\rightharpoonup}\otimes H\to H$, we verify the compatibility condition:
\[
\begin{split}
    \mathrm{Ad}_{L}(a\rightharpoonup b&)=(a_{1}\rightharpoonup b_{1})\bullet_{\rightharpoonup}S_{\rightharpoonup}(a_{3}\rightharpoonup b_{3})\otimes(a_{2}\rightharpoonup b_{2})\\&\overset{\eqref{Sharponharp}}{=}(a_{1}\rightharpoonup b_{1})\bullet_{\rightharpoonup}(a_{3}\leftharpoonup b_{3})\bullet_{\rightharpoonup}S_{\rightharpoonup}(b_{4})\bullet_{\rightharpoonup}S_{\rightharpoonup}(a_{4})\otimes(a_{2}\rightharpoonup b_{2})\\&\overset{\eqref{equivalentcondition}}{=}a_{1}\bullet_{\rightharpoonup}b_{1}\bullet_{\rightharpoonup}S_{\rightharpoonup}(b_{3})\bullet_{\rightharpoonup}S_{\rightharpoonup}(a_{3})\otimes(a_{2}\rightharpoonup b_{2}).
\end{split}
\]
Hence we obtain the braiding operator $\sigma^{\mathcal{YD}}_{H,H}:H\otimes H\to H\otimes H$ given by
\[
\sigma^{\mathcal{YD}}_{H,H}:a\otimes b\mapsto(a_{1}\bullet_{\rightharpoonup}S_{\rightharpoonup}(a_{3})\rightharpoonup b)\otimes a_{2}.
\]
Let us observe that
\[
(a_{1}\bullet_{\rightharpoonup}S_{\rightharpoonup}(a_{3})\rightharpoonup b)\otimes a_{2}=(a_{1}\rightharpoonup(S_{\rightharpoonup}(a_{3})\rightharpoonup b))\otimes a_{2}=\alpha_{\rightharpoonup}a_{1}(\alpha_{\rightharpoonup}S_{\rightharpoonup}(a_{3})(b))\otimes a_{2}\overset{\eqref{betarightharp}}{=}\alpha_{\rightharpoonup}a_{1}(\beta_{\rightharpoonup}a_{3}(b))\otimes a_{2}.
\]
We compute
\[
\begin{split}
    (m_{\cdot}\otimes m_{\cdot})(\mathrm{Id}_{H}\otimes\sigma^{\mathcal{YD}}_{H,H}\otimes\mathrm{Id}_{H})(\Delta\otimes\Delta)(a\otimes b)&
=\big(a_{1}\cdot(\alpha_{\rightharpoonup}a_{2}(\beta_{\rightharpoonup}a_{4}(b_{1})))\big)\otimes\big(a_{3}\cdot b_{2}\big)\overset{\eqref{comp.Deltadot}}{=}\Delta m_{\cdot}(a\otimes b),
\end{split}
\]
thus, since $\epsilon(a\cdot b)=\epsilon(a)\epsilon(b)$, $\epsilon(1_{H})=1_{\Bbbk}$ and $\Delta(1_{H})=1_{H}\otimes1_{H}$ hold by definition of Yetter--Drinfeld post-Hopf algebra, the compatibility condition for a bialgebra in $^{H_{\rightharpoonup}}_{H_{\rightharpoonup}}\mathcal{YD}$ is satisfied. By Remark \ref{Hdotmoncomon} we already know that $(H,\cdot,1)$ is in $\mathrm{Mon}(_{H_{\rightharpoonup}}\mm)$ and $(H,\Delta,\epsilon)$ is in $\mathrm{Comon}(_{H_{\rightharpoonup}}\mm)$.
Moreover, clearly, $u,\Delta$ and $\epsilon$ are morphisms in $^{H_{\rightharpoonup}}\mathfrak{M}$ with respect to the adjoint coaction $\mathrm{Ad}_{L}$. Hence, we show that also $m_{\cdot}$ is left colinear in order to obtain that $(H,\cdot,1,\Delta,\epsilon)$ is in $\mathrm{Bimon}(^{H_{\rightharpoonup}}_{H_{\rightharpoonup}}\mathcal{YD})$. Let us first observe that 
\[
\begin{split}
(\mathrm{Id}\otimes\Delta)\Delta(a\cdot b)&=\big(a_{1}\cdot(\alpha_{\rightharpoonup}a_{2}(\beta_{\rightharpoonup}a_{4}(b_{1})))\big)\otimes\Delta(a_{3}\cdot b_{2})\\&=\big(a_{1}\cdot(\alpha_{\rightharpoonup}a_{2}(\beta_{\rightharpoonup}a_{7}(b_{1})))\big)\otimes(a_{3}\cdot(\alpha_{\rightharpoonup}a_{4}(\beta_{\rightharpoonup}a_{6}(b_{2}))))\otimes(a_{5}\cdot b_{3}).
\end{split}
\]
From \eqref{equivalentcondition} we obtain that
\[
(S_{\rightharpoonup}(a_{3})\rightharpoonup b_{1})\bullet_{\rightharpoonup}(S_{\rightharpoonup}(a_{1})\leftharpoonup b_{3})\otimes(S_{\rightharpoonup}(a_{2})\rightharpoonup b_{2})=(S_{\rightharpoonup}(a_{2})\bullet_{\rightharpoonup}b_{1})\otimes(S_{\rightharpoonup}(a_{1})\rightharpoonup b_{2}),
\]
which is equivalent to
\[
\begin{split}
\big(a_{1}\bullet_{\rightharpoonup}&(S_{\rightharpoonup}(a_{5})\rightharpoonup b_{1})\bullet_{\rightharpoonup}(S_{\rightharpoonup}(a_{3})\leftharpoonup b_{3})\bullet_{\rightharpoonup}S_{\rightharpoonup}(b_{4})\big)\otimes\big(a_{2}\bullet_{\rightharpoonup}(S_{\rightharpoonup}(a_{4})\rightharpoonup b_{2})\big)\\&\ =\big(a_{1}\bullet_{\rightharpoonup}S_{\rightharpoonup}(a_{4})\bullet_{\rightharpoonup}b_{1}\bullet_{\rightharpoonup}S_{\rightharpoonup}(b_{3})\big)\otimes\big(a_{2}\bullet_{\rightharpoonup}(S_{\rightharpoonup}(a_{3})\rightharpoonup b_{2})\big)\\&\overset{\eqref{dotfrombullet}}{=}\big(a_{1}\bullet_{\rightharpoonup}S_{\rightharpoonup}(a_{3})\bullet_{\rightharpoonup}b_{1}\bullet_{\rightharpoonup}S_{\rightharpoonup}(b_{3})\big)\otimes\big(a_{2}\cdot b_{2}\big).
\end{split}
\]
Moreover, we have
\[
\begin{split}
&\big(a_{1}\bullet_{\rightharpoonup}(S_{\rightharpoonup}(a_{5})\rightharpoonup b_{1})\bullet_{\rightharpoonup}(S_{\rightharpoonup}(a_{3})\leftharpoonup b_{3})\bullet_{\rightharpoonup}S_{\rightharpoonup}(b_{4})\big)\otimes\big(a_{2}\bullet_{\rightharpoonup}(S_{\rightharpoonup}(a_{4})\rightharpoonup b_{2})\big)\\&\overset{\eqref{defleftharp}}{=}\big(a_{1}\bullet_{\rightharpoonup}(S_{\rightharpoonup}(a_{6})\rightharpoonup b_{1})\bullet_{\rightharpoonup}S_{\rightharpoonup}(S_{\rightharpoonup}(a_{4})\rightharpoonup b_{3})\bullet_{\rightharpoonup}S_{\rightharpoonup}(a_{3})\bullet_{\rightharpoonup}b_{4}\bullet S_{\rightharpoonup}(b_{5})\big)\otimes\big(a_{2}\bullet_{\rightharpoonup}(S_{\rightharpoonup}(a_{5})\rightharpoonup b_{2})\big)\\&\ =\big(a_{1}\bullet_{\rightharpoonup}(S_{\rightharpoonup}(a_{6})\rightharpoonup b_{1})\bullet_{\rightharpoonup}S_{\rightharpoonup}(S_{\rightharpoonup}(a_{4})\rightharpoonup b_{3})\bullet_{\rightharpoonup}S_{\rightharpoonup}(a_{3})\big)\otimes\big(a_{2}\bullet_{\rightharpoonup}(S_{\rightharpoonup}(a_{5})\rightharpoonup b_{2})\big)\\&\ =\big(a_{1}\bullet_{\rightharpoonup}(S_{\rightharpoonup}(a_{6})\rightharpoonup b_{1})\bullet_{\rightharpoonup}S_{\rightharpoonup}(a_{3}\bullet_{\rightharpoonup}(S_{\rightharpoonup}(a_{4})\rightharpoonup b_{3}))\big)\otimes\big(a_{2}\bullet_{\rightharpoonup}(S_{\rightharpoonup}(a_{5})\rightharpoonup b_{2})\big)\\&\ =\big(a_{1}\bullet_{\rightharpoonup}(S_{\rightharpoonup}(a_{2})\rightharpoonup b)\big)_{1}\bullet_{\rightharpoonup}S_{\rightharpoonup}\Big(\big(a_{1}\bullet_{\rightharpoonup}(S_{\rightharpoonup}(a_{2})\rightharpoonup b)\big)_{3}\Big)\otimes\big(a_{1}\bullet_{\rightharpoonup}(S_{\rightharpoonup}(a_{2})\rightharpoonup b)\big)_{2}\\&\overset{\eqref{dotfrombullet}}{=}(a\cdot b)_{1}\bullet_{\rightharpoonup}S_{\rightharpoonup}((a\cdot b)_{3})\otimes(a\cdot b)_{2},
\end{split}
\]
hence we obtain
\[
(a\cdot b)_{1}\bullet_{\rightharpoonup}S_{\rightharpoonup}((a\cdot b)_{3})\otimes(a\cdot b)_{2}=\big(a_{1}\bullet_{\rightharpoonup}S_{\rightharpoonup}(a_{3})\bullet_{\rightharpoonup}b_{1}\bullet_{\rightharpoonup}S_{\rightharpoonup}(b_{3})\big)\otimes\big(a_{2}\cdot b_{2}\big),
\]
i.e. $m_{\cdot}$ is left colinear. Thus, $(H,\cdot,1,\Delta,\epsilon)$ is in $\mathrm{Bimon}(^{H_{\rightharpoonup}}_{H_{\rightharpoonup}}\mathcal{YD})$. Finally, we already know that $S$ satisfies $a_{1}\cdot S(a_{2})=\epsilon(a)1_{H}=S(a_{1})\cdot a_{2}$ for all $a\in H$, hence $(H,\cdot,1,\Delta,\epsilon,S)$ is in $\mathrm{Hopf}(^{H_{\rightharpoonup}}_{H_{\rightharpoonup}}\mathcal{YD})$ by \cite[Proposition 3.10]{FS}.
\end{proof}

\begin{remark}
    If $H$ is cocommutative the braiding $\sigma^{\mathcal{YD}}_{H,H}:a\otimes b\mapsto\alpha_{\rightharpoonup}a_{1}(\beta_{\rightharpoonup}a_{3}(b))\otimes a_{2}$ becomes the canonical flip $\tau_{H,H}$ and $(H,\cdot,1,\Delta,\epsilon,S)$ becomes a standard Hopf algebra.
\end{remark}

We can prove the following result:

\begin{proposition}
    Let $(H,\cdot,1,\Delta,\epsilon,S,\rightharpoonup)$ be a Yetter--Drinfeld post-Hopf algebra. Then, the set of group-like elements $G(H)$ is a post-group.
\end{proposition}

\begin{proof}
We know that $\Delta(1_{H})=1_{H}\ot1_{H}$ and $\epsilon(1_{H})=1_{\Bbbk}$, so $1_{H}\in G(H)$.
In order to obtain that $(G(H),\cdot)$ is a group, we need to prove that $G(H)$ is closed under $\cdot$ and $S$. Indeed, we already know that $x\cdot S(x)=S(x)\cdot x=1_{H}$ for all $x\in G(H)$. Since $\alpha_{\rightharpoonup}:H\to\mathrm{End}(H)$ is convolution invertible with convolution inverse $\beta_{\rightharpoonup}$, we obtain $(\alpha_{\rightharpoonup}x)\circ(\beta_{\rightharpoonup}x)=(\beta_{\rightharpoonup}x)\circ(\alpha_{\rightharpoonup}x)=\mathrm{Id}$ for all $x\in G(H)$ and then, for all $x,y\in G(H)$, we have 
\[
\Delta(x\cdot y)\overset{\eqref{comp.Deltadot}}{=}\Big(x\cdot\big(\alpha_{\rightharpoonup}x(\beta_{\rightharpoonup}x(y))\big)\Big)\ot(x\cdot y)=(x\cdot y)\ot(x\cdot y), \qquad \epsilon(x\cdot y)=\epsilon(x)\epsilon(y)=1_{\Bbbk},
\]
hence $x\cdot y\in G(H)$. Moreover, by Theorem \ref{HdotHopfYD}, we know that $(H,\cdot,1,\Delta,\epsilon,S)$ is in $\mathrm{Hopf}(^{H_{\rightharpoonup}}_{H_{\rightharpoonup}}\mathcal{YD})$, hence $S$ satisfies $(S\ot S)\sigma^{\mathcal{YD}}_{H,H}\Delta=\Delta S$ and $\epsilon S=\epsilon$. Hence, given $x\in G(H)$, we obtain
\[
(S\ot S)\sigma^{\mathcal{YD}}_{H,H}\Delta(x)=(S\ot S)\sigma^{\mathcal{YD}}_{H,H}(x\ot x)=(S\ot S)\big(\alpha_{\rightharpoonup}x(\beta_{\rightharpoonup}x(x))\ot x\big)=S(x)\ot S(x),
\]
so $\Delta(S(x))=S(x)\ot S(x)$ and $\epsilon(S(x))=\epsilon(x)=1_{\Bbbk}$ and then $S(x)\in G(H)$. Thus, $(G(H),\cdot)$ is a group. Moreover, since $\rightharpoonup$ is a coalgebra morphism, for all $x,y\in G(H)$ we have
\[
\Delta(x\rightharpoonup y)=(x_{1}\rightharpoonup y_{1})\ot(x_{2}\rightharpoonup y_{2})=(x\rightharpoonup y)\ot(x\rightharpoonup y), \qquad \epsilon(x\rightharpoonup y)=\epsilon(x)\epsilon(y)=1_{\Bbbk},
\]
hence we obtain a map $\rightharpoonup:G(H)\times G(H)\to G(H)$. Since \eqref{dotlinear} and \eqref{deformedassociativityright} are satisfied for all $x,y,z\in H$, we have 
\[
x\rightharpoonup(y\cdot z)=(x\rightharpoonup y)\cdot(x\rightharpoonup z),\qquad x\rightharpoonup(y\rightharpoonup z)=(x\cdot(x\rightharpoonup y))\rightharpoonup z 
\]
for all $x,y,z\in G(H)$. In particular, $\alpha_{\rightharpoonup}x:G(H)\to G(H)$ is an automorphism of the group $G(H)$ for all $x\in G(H)$ and then $(G(H),\cdot,\rightharpoonup)$ is a post-group.
\end{proof}

Since a Yetter--Drinfeld post-Hopf algebra $(H,\rightharpoonup)$ is an object in $\mathrm{Mon}(^{H_{\rightharpoonup}}_{H_{\rightharpoonup}}\mathcal{YD})$ we call it a \textbf{Yetter--Drinfeld pre-Hopf algebra} is $H$ is braided commutative, i.e. $m_{\cdot}\circ\sigma^{\mathcal{YD}}_{H,H}=m_{\cdot}$, which on elements $a,b\in H$ reads: 
\[
\alpha_{\rightharpoonup}a_{1}(\beta_{\rightharpoonup}a_{3}(b))\cdot a_{2}=a\cdot b.
\]
If $H$ is cocommutative then $\sigma^{\mathcal{YD}}_{H,H}=\tau_{H,H}$ and the braided commutativity becomes the classical commutativity.

\begin{remark}
    Since a Yetter--Drinfeld post-Hopf algebra $(H,\rightharpoonup)$ is an object in $\mathrm{Hopf}(^{H_{\rightharpoonup}}_{H_{\rightharpoonup}}\mathcal{YD})$, one can consider the braided commutator $[a,b]_{\mathcal{YD}}:=a\cdot b-m_{\cdot}\sigma^{\mathcal{YD}}_{H,H}(a\otimes b)$ studying primitive elements. But if  $a\in P(H)$, recalling that $\alpha_{\rightharpoonup}1=\beta_{\rightharpoonup}1=\mathrm{Id}_{H}$, we obtain 
\[
\begin{split}
\sigma^{\mathcal{YD}}_{H,H}(a\otimes b&)=\alpha_{\rightharpoonup}1(\beta_{\rightharpoonup}a(b))\otimes1+\alpha_{\rightharpoonup}1(\beta_{\rightharpoonup}1(b))\otimes a+\alpha_{\rightharpoonup}a(\beta_{\rightharpoonup}1(b))\otimes 1\\&\overset{\eqref{betarightharp}}{=}\alpha_{\rightharpoonup}S_{\rightharpoonup}a(b)\otimes 1+b\otimes a+\alpha_{\rightharpoonup}a(b)\otimes1\\&\ =-\alpha_{\rightharpoonup}a(b)\otimes1+b\otimes a+\alpha_{\rightharpoonup}a(b)\otimes 1\\&=b\otimes a,
\end{split}
\]
hence the braided commutator coincides with the classical one.
\end{remark}

Finally, we obtain that any Yetter--Drinfeld post-Hopf algebra provides a Yetter--Drinfeld brace.

\begin{corollary}\label{cor:fromPHtoBr}
    Let $(H,\rightharpoonup)$ be a Yetter--Drinfeld post-Hopf algebra. Then, $(H,\cdot,\bullet_{\rightharpoonup},1,\Delta,\epsilon,S,S_{\rightharpoonup})$ is a Yetter--Drinfeld brace. This yields a functor $F:\mathcal{YD}\mathrm{PH}(\mathrm{Vec}_{\Bbbk})\to\mathcal{YD}\mathrm{Br}(\mathrm{Vec}_{\Bbbk})$.
\end{corollary}

\begin{proof}
    We know that $(H,\bullet_{\rightharpoonup},1,\Delta,\epsilon,S_{\rightharpoonup})$ is a Hopf algebra by Proposition \ref{HharpHopf} and $(H,\cdot,1,\Delta,\epsilon,S)$ is in $\mathrm{Hopf}(^{H_{\rightharpoonup}}_{H_{\rightharpoonup}}\mathcal{YD})$ by Theorem \ref{HdotHopfYD}, where the coaction is given by $\mathrm{Ad}_{L}$. Moreover, by \eqref{harpfromS} we know that the action $\rightharpoonup$ is given by $a\rightharpoonup b=S(a_{1})\cdot(a_{2}\bullet_{\rightharpoonup}b)$ and, 
if we define $\leftharpoonup$ as in \eqref{defleftharp}, the pair $(\rightharpoonup,\leftharpoonup)$ satisfies \eqref{MP5}. In order to conclude that $(H,\cdot,\bullet_{\rightharpoonup},1,\Delta,\epsilon,S,S_{\rightharpoonup})$ is a Yetter--Drinfeld brace it remains to prove that $\cdot$ and $\bullet_{\rightharpoonup}$ satisfy the compatibility of Hopf braces:
\[
\begin{split}
    x\bullet_{\rightharpoonup}(y\cdot z)&=x_{1}\cdot(x_{2}\rightharpoonup(y\cdot z))\overset{\eqref{dotlinear}}{=}x_{1}\cdot((x_{2}\rightharpoonup y)\cdot(x_{3}\rightharpoonup z))=x_{1}\cdot(x_{2}\rightharpoonup y)\cdot S(x_{3})\cdot x_{4}\cdot(x_{5}\rightharpoonup z)\\&=(x_{1}\bullet_{\rightharpoonup}y)\cdot S(x_{2})\cdot(x_{3}\bullet_{\rightharpoonup}z).
\end{split}
\]
Let $f:(H,\cdot,1,\Delta,\epsilon,S
,\rightharpoonup)\to(H',\cdot',1',\Delta',\epsilon',S'
,\rightharpoonup')$ be a morphism of Yetter--Drinfeld post-Hopf algebras. We already know that $f
$ is a morphism of algebras and coalgebras, so we just have to show that it is multiplicative with respect to $\bullet_{\rightharpoonup}$ in order to obtain that it is a morphism of Yetter--Drinfeld braces $f:(H,\cdot,\bullet_{\rightharpoonup},1,\Delta,\epsilon,S,S_{\rightharpoonup})\to(H',\cdot',\bullet_{\rightharpoonup'}',1',\Delta',\epsilon',S',S'_{\rightharpoonup'})$. We compute
\[
\begin{split}
f(a\bullet_{\rightharpoonup}b)=f(a_{1}\cdot(a_{2}\rightharpoonup b))&=f(a_{1})\cdot f(a_{2}\rightharpoonup b)\overset{\eqref{morphismYDpost}}{=}f(a_{1})\cdot(f(a_{2})\rightharpoonup'f(b))=f(a)_{1}\cdot(f(a)_{2}\rightharpoonup'f(b))\\&=f(a)\bullet_{\rightharpoonup'}f(b).
\end{split}
\]
Clearly the assignment $F$ respects identities and compositions.
\end{proof}

We can also prove that any Yetter--Drinfeld brace provides a Yetter--Drinfeld post-Hopf algebra as follows.

\begin{proposition}\label{prop:fromYDbracetoYDpostHopf}
Let $(H,\cdot,\bullet,1,\Delta,\epsilon,S,T)$ be a Yetter--Drinfeld brace. Then $(H,\cdot,1,\Delta,\epsilon,S,\rightharpoonup)$ is a Yetter--Drinfeld post-Hopf algebra with $\rightharpoonup$ defined by
\[
a\rightharpoonup b:=S(a_{1})\cdot(a_{2}\bullet b)
\]
and $\beta_{\rightharpoonup}a:b\mapsto(T(a)\rightharpoonup b)$, for all $a,b\in H$. This yields a functor $G:\mathcal{YD}\mathrm{Br}(\mathrm{Vec}_{\Bbbk})\to\mathcal{YD}\mathrm{PH}(\mathrm{Vec}_{\Bbbk})$.
\end{proposition}

\begin{proof}
Since $(H,\cdot,\bullet,1,\Delta,\epsilon,S,T)$ is a Yetter--Drinfeld brace, by definition $(H,\cdot,1,\Delta,\epsilon,S)$ is in $\mathrm{Hopf}(^{H^{\bullet}}_{H^{\bullet}}\mathcal{YD})$, where $H^{\bullet}:=(H,\bullet,1,\Delta,\epsilon,T)$ is a Hopf algebra, the $H^{\bullet}$-action and the $H^{\bullet}$-coaction on $H$ are given by $a\rightharpoonup b:=S(a_{1})\cdot(a_{2}\bullet b)$ and $\mathrm{Ad}_{L}:a\mapsto a_{1}\bullet T(a_{3})\otimes a_{2}$, respectively. In particular, $(H,\cdot,1)$ is an algebra, $(H,\Delta,\epsilon)$ is a coalgebra, $S$ satisfies $a_{1}\cdot S(a_{2})=\epsilon(a)1=S(a_{1})\cdot a_{2}$ and $\rightharpoonup$ is a morphism of coalgebras which satisfies \eqref{dotlinear}. From $a_{1}\cdot(a_{2}\rightharpoonup b)=a_{1}\cdot S(a_{2})\cdot(a_{3}\bullet b)=a\bullet b$, also \eqref{deformedassociativityright} is automatically true as $\rightharpoonup$ is an action. Define $(\beta_{\rightharpoonup}a)(b):=T(a)\rightharpoonup b$, for all $a,b\in H$. Hence we have 
\[
\beta_{\rightharpoonup}a_{1}(\alpha_{\rightharpoonup}a_{2}(b))=
T(a_{1})\rightharpoonup(a_{2}\rightharpoonup b)=(T(a_{1})\bullet a_{2})\rightharpoonup b=\epsilon(a)1\rightharpoonup b=\epsilon(a)b, 
\]
and, similarly, $\alpha_{\rightharpoonup}a_{1}(\beta_{\rightharpoonup}a_{2}(b))=\epsilon(a)b$.
Hence $\alpha_{\rightharpoonup}$ is convolution invertible in $\mathrm{Hom}(H,\mathrm{End}(H))$ with inverse given by $\beta_{\rightharpoonup}$. Moreover, we already know that $\epsilon$ and $\Delta$ are morphisms of algebras in $^{H^{\bullet}}_{H^{\bullet}}\mathcal{YD}$, i.e. $\epsilon(a\cdot b)=\epsilon(a)\epsilon(b)$, $\epsilon(1_{H})=1_{\Bbbk}$, $\Delta(1_{H})=1_{H}\otimes1_{H}$ and the following compatibility condition holds true:
\begin{equation}\label{compcondiYD}
\Delta(a\cdot b)=a_{1}\cdot(a_{2}\bullet T(a_{4})\rightharpoonup b_{1})\otimes(a_{3}\cdot b_{2}).
\end{equation}
Hence we have
\[
\Delta(a\cdot b)\overset{\eqref{compcondiYD}}{=}a_{1}\cdot(a_{2}\rightharpoonup(T(a_{4})\rightharpoonup b_{1}))\otimes(a_{3}\cdot b_{2})
=\big(a_{1}\cdot(\alpha_{\rightharpoonup}a_{2}(\beta_{\rightharpoonup}a_{4}(b_{1})))\big)\otimes\big(a_{3}\cdot b_{2}\big),
\]
i.e. \eqref{comp.Deltadot} is satisfied.
Finally, $a\bullet_{\rightharpoonup}b:=a_{1}\cdot(a_{2}\rightharpoonup b)=a\bullet b$ and, since $a_{1}\rightharpoonup T(a_{2})=S(a_{1})\cdot(a_{2}\bullet T(a_{3}))=S(a)$, we obtain
\[
S_{\rightharpoonup}(a)\overset{\eqref{defantipode}}{=}\beta_{\rightharpoonup}a_{1}(S(a_{2}))=T(a_{1})\rightharpoonup(a_{2}\rightharpoonup T(a_{3}))=(T(a_{1})\bullet a_{2})\rightharpoonup T(a_{3})=T(a).
\]
Hence $S_{\rightharpoonup}$ is anti-comultiplicative and $a\leftharpoonup b:=S_{\rightharpoonup}(a_{1}\rightharpoonup b_{1})\bullet_{\rightharpoonup}a_{2}\bullet_{\rightharpoonup}b_{2}=T(a_{1}\rightharpoonup b_{1})\bullet a_{2}\bullet b_{2}$. Thus, by definition of Yetter--Drinfeld brace, the pair $(\rightharpoonup,\leftharpoonup)$ satisfies \eqref{MP5} and $(H,\rightharpoonup)$ is a Yetter--Drinfeld post-Hopf algebra. 

Let $f:(H,\cdot,\bullet,1,\Delta,\epsilon,S,T)\to(H',\cdot',\bullet',1',\Delta',\epsilon',S',T')$ be a morphism of Yetter--Drinfeld braces, so $f:(H,\bullet,1,\Delta,\epsilon,T)\to(H',\bullet',1',\Delta',\epsilon',T')$ is a morphism of Hopf algebras and satisfies $f(a\cdot b)=f(a)\cdot'f(b)$ and $fS=S'f$. We prove that $f(a\rightharpoonup b)=f(a)\rightharpoonup'f(b)$, so that $f:(H,\cdot,1,\Delta,\epsilon,S,\rightharpoonup)\to(H',\cdot',1',\Delta',\epsilon',S',\rightharpoonup')$ is a morphism of Yetter--Drinfeld post-Hopf algebras:
\[
\begin{split}
f(a\rightharpoonup b)&=f(S(a_{1})\cdot(a_{2}\bullet b))=f(S(a_{1}))\cdot' f(a_{2}\bullet b)=S'(f(a_{1}))\cdot'(f(a_{2})\bullet' f(b))\\&=S'(f(a)_{1})\cdot'(f(a)_{2}\bullet' f(b))=f(a)\rightharpoonup'f(b).
\end{split}
\]
Clearly the assignment $G$ respects identities and compositions.
\end{proof}

\begin{remark}
    The previous proof also shows that the subadjacent Hopf algebra of the Yetter--Drinfeld post-Hopf algebra $(H,\cdot,1,\Delta,\epsilon,S,\rightharpoonup)$, where the latter is obtained from the Yetter--Drinfeld brace $(H,\cdot,\bullet,1,\Delta,\epsilon,S,T)$, is given by the Hopf algebra $(H,\bullet,1,\Delta,\epsilon,T)$. Then, clearly $FG=\mathrm{Id}$. Moreover, given a Yetter--Drinfeld post-Hopf algebra $(H,\cdot,1,\Delta,\epsilon,S\rightharpoonup)$ and the Yetter--Drinfeld brace $F((H,\rightharpoonup))=(H,\cdot,\bullet_{\rightharpoonup},1,\Delta,\epsilon,S,S_{\rightharpoonup})$, then $GF((H,\rightharpoonup))$ is the Yetter--Drinfeld post-Hopf algebra $(H,\cdot,1,\Delta,\epsilon,S,\rightharpoonup')$ where
\[
a\rightharpoonup'b:=S(a_{1})\cdot(a_{2}\bullet_{\rightharpoonup}b)=S(a_{1})\cdot a_{2}\cdot(a_{3}\rightharpoonup b)=a\rightharpoonup b.
\]
Hence $GF=\mathrm{Id}$ and we obtain the following result.
\end{remark}

\begin{corollary}\label{YDPHisoYDBr}
    The categories $\mathcal{YD}\mathrm{PH}(\mathrm{Vec}_{\Bbbk})$ and $\mathcal{YD}\mathrm{Br}(\mathrm{Vec}_{\Bbbk})$ are isomorphic. 
\end{corollary}

By \cite[Theorem 3.25]{FS} we know that the category $\mathcal{YD}\mathrm{Br}(\mathrm{Vec}_{\Bbbk})$ is isomorphic to the category $\mathrm{MP}(\mathrm{Vec}_{\Bbbk})$ of matched pairs of actions. Thus we obtain the following result.

\begin{corollary}\label{cor:fromPHtoMP}
    The categories $\mathcal{YD}\mathrm{PH}(\mathrm{Vec}_{\Bbbk})$ and $\mathrm{MP}(\mathrm{Vec}_{\Bbbk})$ are isomorphic. Explicitly, from a Yetter--Drinfeld post-Hopf algebra $(H,\cdot,1,\Delta,\epsilon,S,\rightharpoonup)$ one obtains a matched pair of actions $(H_{\rightharpoonup},\rightharpoonup,\leftharpoonup)$ on the subadjacent Hopf algebra $H_{\rightharpoonup}$, where $\leftharpoonup$ is defined as in \eqref{defleftharp}. Conversely, given a matched pair of actions $(H,\rightharpoonup,\leftharpoonup)$ on a Hopf algebra $(H,\bullet,1,\Delta,\epsilon,T)$ one obtains a Yetter--Drinfeld post-Hopf algebra $(H,\cdot,1,\Delta,\epsilon,S,\rightharpoonup)$ where $a\cdot b:=a_{1}\bullet(T(a_{2})\rightharpoonup b)$ and $S(a):=a_{1}\rightharpoonup T(a_{2})$.
\end{corollary}

The previous result generalises \cite[Corollary 4.5]{YYT}. \medskip

To end this subsection we make a connection with braiding operators. It is known that matched pairs of actions on cocommutative Hopf algebras, braiding operators on cocommutative Hopf algebras and cocommutative Hopf braces are equivalent \cite{GGVe}. Moreover, in \cite[Theorem 6.34]{GGV} the category of Yetter--Drinfeld braces is proven to be isomorphic to the category 
of braiding operators on Hopf algebras \cite[Definition 5.1]{GGV}. Hence, from Corollary \ref{YDPHisoYDBr}, we obtain the following:

\begin{corollary}
    The category $\mathcal{YD}\mathrm{PH}(\mathrm{Vec}_{\Bbbk})$ is isomorphic to the category of braiding operators on Hopf algebras.
\end{corollary}

Let us recall the definition of \textit{braiding operator on a Hopf algebra} $H$ as stated in \cite[Definition 3.1]{Li}. It consists of a coalgebra morphism $r:H\ot H\to H\ot H$ such that:
\begin{itemize}
    \item[i)] $mr=m$, \medskip
    \item[ii)] $r(m\ot\mathrm{Id})=(\mathrm{Id}\ot m)(r\ot\mathrm{Id})(\mathrm{Id}\ot r)$,\medskip
    \item[iii)] $r(\mathrm{Id}\ot m)=(m\ot\mathrm{Id})(\mathrm{Id}\ot r)(r\ot\mathrm{Id})$,\medskip
    \item[iv)] $r(u\ot\mathrm{Id})=\mathrm{Id}\ot u$, \medskip
    \item[v)] $r(\mathrm{Id}\ot u)=u\ot\mathrm{Id}$,
\end{itemize}
where $m$ and $u$ denote the multiplication and the unit of $H$, respectively. 

In \cite[Theorem 5.27]{GGV} it is proved that braiding operators on Hopf algebras provide solutions of the braid equation. Thus, as it is shown in \cite[Theorem 3.2]{Li}, matched pairs of actions on Hopf algebras $(H,\rightharpoonup,\leftharpoonup)$ provide solutions of the braid equation $r(x\ot y):=(x_{1}\rightharpoonup y_{1})\ot(x_{2}\leftharpoonup y_{2})$. This result extends \cite[Corollary 2.4]{AGV}, obtained under the cocommutativity assumption. 

Therefore, from Corollary \ref{cor:fromPHtoMP}, we obtain:

\begin{proposition}
Let $(H,\cdot,1,\Delta,\epsilon,S,\rightharpoonup)$ be a Yetter--Drinfeld post-Hopf algebra. Then 
\begin{equation}\label{eq:solbraideq}
r:H\ot H\to H\ot H,\ x\ot y\mapsto (x_{1}\rightharpoonup y_{1})\ot(x_{2}\leftharpoonup y_{2})
\end{equation}
is a solution of the braid equation, where $\leftharpoonup$ is defined as in \eqref{defleftharp}. More explicitly, we have
\[
r(x\ot y)=(x_{1}\rightharpoonup y_{1})\ot\Big(\beta_{\rightharpoonup}(x_{4}\rightharpoonup y_{4})(x_{5}\rightharpoonup S(y_{5}))\cdot \beta_{\rightharpoonup}(x_{3}\rightharpoonup y_{3})( x_{6})\cdot \beta_{\rightharpoonup}(x_{2}\rightharpoonup y_{2})(x_{7}\rightharpoonup y_{6})\Big),
\]
where $\beta_{\rightharpoonup}$ is the convolution inverse of $\alpha_{\rightharpoonup}$.
\end{proposition}

\begin{proof}
    We compute
\[
\begin{split}
(x_{1}\rightharpoonup y_{1})&\ot(x_{2}\leftharpoonup y_{2})\overset{\eqref{defleftharp}}{=}(x_{1}\rightharpoonup y_{1})\ot\big(S_{\rightharpoonup}(x_{2}\rightharpoonup y_{2})\bullet_{\rightharpoonup}x_{3}\bullet_{\rightharpoonup}y_{3}\big)\\&\overset{\eqref{defbullet}}{=}(x_{1}\rightharpoonup y_{1})\ot\Big(S_{\rightharpoonup}(x_{3}\rightharpoonup y_{3})\cdot\big(S_{\rightharpoonup}(x_{2}\rightharpoonup y_{2})\rightharpoonup(x_{4}\cdot(x_{5}\rightharpoonup y_{4}))\big)\Big)\\&\overset{\eqref{dotlinear}}{=}(x_{1}\rightharpoonup y_{1})\ot\Big(S_{\rightharpoonup}(x_{4}\rightharpoonup y_{4})\cdot(S_{\rightharpoonup}(x_{3}\rightharpoonup y_{3})\rightharpoonup x_{5})\cdot\big(S_{\rightharpoonup}(x_{2}\rightharpoonup y_{2})\rightharpoonup(x_{6}\rightharpoonup y_{5})\big)\Big)\\&\ =(x_{1}\rightharpoonup y_{1})\ot\Big(S_{\rightharpoonup}(x_{4}\rightharpoonup y_{4})\cdot \alpha_{\rightharpoonup}S_{\rightharpoonup}(x_{3}\rightharpoonup y_{3})( x_{5})\cdot \alpha_{\rightharpoonup}S_{\rightharpoonup}(x_{2}\rightharpoonup y_{2})(x_{6}\rightharpoonup y_{5})\Big)\\&\overset{\eqref{betarightharp}}{=}(x_{1}\rightharpoonup y_{1})\ot\Big(S_{\rightharpoonup}(x_{4}\rightharpoonup y_{4})\cdot \beta_{\rightharpoonup}(x_{3}\rightharpoonup y_{3})( x_{5})\cdot \beta_{\rightharpoonup}(x_{2}\rightharpoonup y_{2})(x_{6}\rightharpoonup y_{5})\Big)\\&\overset{\eqref{defantipode},\eqref{Slinear}}{=}(x_{1}\rightharpoonup y_{1})\ot\Big(\beta_{\rightharpoonup}(x_{4}\rightharpoonup y_{4})(x_{5}\rightharpoonup S(y_{5}))\cdot \beta_{\rightharpoonup}(x_{3}\rightharpoonup y_{3})( x_{6})\cdot \beta_{\rightharpoonup}(x_{2}\rightharpoonup y_{2})(x_{7}\rightharpoonup y_{6})\Big)
\end{split}
\]
where $\beta_{\rightharpoonup}$ is the convolution inverse of $\alpha_{\rightharpoonup}$.
\end{proof}

\subsection{Examples of Yetter--Drinfeld post-Hopf algebras}

By Proposition \ref{prop:fromYDbracetoYDpostHopf} we can use Yetter--Drinfeld braces to provide examples of Yetter--Drinfeld post-Hopf algebras. In \cite[Section 6]{FS} some examples of Yetter--Drinfeld braces are given, using matched pairs of actions coming from coquasitriangular structures $\mathcal{R}$ (we refer the reader to \cite{Majid2} for coquasitriangular Hopf algebras). In this case, as it is said in \cite[Remark 5.7]{FS}, the structure $(H,\cdot,1,\Delta,\epsilon,S)$ results to be the transmutation of $(H,\bullet,1,\Delta,\epsilon,T,\mathcal{R})$, see \cite[dual of Example 9.4.10]{Majid2}, \cite{Majid1}. Hence, translating \cite[Theorem 5.6]{FS} into our setting we get the following result.

\begin{lemma}
    Let $(H,\bullet,1,\Delta,\epsilon,T,\mathcal{R})$ be a coquasitriangular Hopf algebra. Define $\cdot$ and $S$ as 
\[
a\cdot b:=a_{1}\bullet b_{2}\mathcal{R}^{-1}(T(a_{2})\otimes b_{1}\bullet T(b_{3})), \qquad S(a):=\mathcal{R}(a_{1}\otimes a_{5})T(a_{4})\mathcal{R}(a_{2}\otimes T(a_{3})),
\]
for all $a,b\in H$, where $\mathcal{R}^{-1}$ denotes the convolution inverse of $\mathcal{R}:H\otimes H\to\Bbbk$. Then, we have that $(H,\cdot,1,\Delta,\epsilon,S,\rightharpoonup)$ is a Yetter--Drinfeld post-Hopf algebra where $\rightharpoonup$ is defined by
\[
a\rightharpoonup b:=b_{2}\mathcal{R}^{-1}(a\otimes b_{1}\bullet T(b_{3}))
\]
and $\beta_{\rightharpoonup}a:b\mapsto b_{2}\Rr^{-1}(T(a)\otimes b_{1}\bullet T(b_{3}))$.
\end{lemma}
In the examples of \cite[Section 6]{FS}, the morphisms $\rightharpoonup$ are explicitly computed, hence we just have to collect the results writing them in terms of the $\cdot$ operation. We will only explain the structures $(H,\cdot,1,\Delta,\epsilon,S,\rightharpoonup)$, saying which standard Hopf algebras they are the transmutation of, without indicating the $\Rr$ structures; the standard Hopf algebras correspond to the subadjacent Hopf algebras. We refer the reader to \cite[Section 6]{FS} for more details above the following examples. 

\begin{example}\label{ex:Sweedler}
Let $\Bbbk$ be a field of $\mathrm{char}(\Bbbk)\not=2$ and the algebra $(H, \cdot,1)$ generated by $g,x$ modulo the relations 
\[
g\cdot g=1,\quad x\cdot x=k1-kg,\quad x\cdot g=g\cdot x,
\]
where $k\in\Bbbk$. Define $\Delta$, $\epsilon$ and $S$ on $g$ and $x$ as
\[
\Delta(g)=g\otimes g,\quad \Delta(x)=x\otimes 1+g\otimes x,\quad \epsilon(g)=1_{\Bbbk},\quad \epsilon(x)=0,\quad S(g)=g,\quad S(x)=-x\cdot g.
\]
We have that $(H,\cdot,1,\Delta,\epsilon,S)$ is the transmutation of the Sweedler's Hopf algebra $H_{4}$, see \cite[6.1]{FS}.
Moreover, $(H,\cdot,1,\Delta,\epsilon,S,\rightharpoonup)$ is a Yetter--Drinfeld post-Hopf algebra, where $\rightharpoonup$ is given as in the following left diagram and $\alpha_{\rightharpoonup}$ has convolution inverse $\beta_{\rightharpoonup}$ explained in the right diagram.
\[
\begin{tabular}{c|cccc}
    $\rightharpoonup$& $1$ & $g$ & $x$ & $x\cdot g$\\ \hline
    $1$ & $1$ & $g$ & $x$ & $x\cdot g$ \\ 
    $g$ & $1$ & $g$ & $-x$ & $-x\cdot g$ \\ 
    $x$ & $0$ & $0$ & $k(1-g)$ & $k(g-1)$ \\ 
    $x\cdot g$ & $0$ & $0$ & $k(g-1)$ & $k(1-g)$ \\ 
\end{tabular}\qquad
\begin{tabular}{c|cccc}
    $\beta_{\rightharpoonup}$& $1$ & $g$ & $x$ & $x\cdot g$\\ \hline
    $1$ & $1$ & $g$ & $x$ & $x\cdot g$ \\ 
    $g$ & $1$ & $g$ & $-x$ & $-x\cdot g$ \\ 
    $x$ & $0$ & $0$ & $k(g-1)$ & $k(1-g)$ \\
    $x\cdot g$ &  $0$ & $0$ & $k(g-1)$ & $k(1-g)$ \\ 
\end{tabular}
\]
Let 
us observe that post-Hopf algebra structures on the Sweedler's Hopf algebra are classified in \cite[Example 2.12]{YYT}. The Sweedler's Hopf algebra is not cocommutative and indeed the post-Hopf algebra structures generally differ from the previous Yetter--Drinfeld post-Hopf algebra structures, in which $H$ is a Hopf monoid in the category of Yetter--Drinfeld modules and not a standard Hopf algebra.


\end{example}

\begin{example}\label{ex:E(n)}
   Let $\Bbbk$ be a field of $\mathrm{char}(\Bbbk)\not=2$. For a fixed $1\leq n\in\mathbb{N}$, consider the algebra $(H,\cdot,1)$ generated by $g,x_{1},\ldots, x_{n}$ subject to the relations 
\[
g\cdot g=1, \quad x_{i}\cdot x_{j}+x_{j}\cdot x_{i}=2A_{ij}(1-g),\quad x_{i}\cdot g=g\cdot x_{i},
\]
where $A=(A_{ij})_{ij}$ is an $n\times n$ symmetric matrix with entries in $\Bbbk$. Define $\Delta$, $\epsilon$ and $S$ on $g,x_{i}$ as
\[
\Delta(g)=g\otimes g,\quad \Delta(x_{i})=x_{i}\otimes1+g\otimes x_{i},\quad \epsilon(g)=1_{\Bbbk},\quad \epsilon(x_{i})=0,\quad S(g)=g, \quad S(x_{i})=-x_{i}\cdot g.
\]
In case $n=1$ we recover the previous example. We have that $(H,\cdot,1,\Delta,\epsilon,S)$ is the transmutation of the Hopf algebra $E(n)$, see \cite[6.2]{FS}. 
Moreover, $(H,\cdot,1,\Delta,\epsilon,S,\rightharpoonup)$ is a Yetter--Drinfeld post-Hopf algebra, where $\rightharpoonup$ and $\beta_{\rightharpoonup}$ are defined by the following tables.
\[
\begin{tabular}{c|cccc}
    $\rightharpoonup$& $1$ & $g$ & $x_{j}$ & $x_{j}\cdot g$\\ \hline
    $1$ & $1$ & $g$ & $x_{j}$ & $x_{j}\cdot g$ \\ 
    $g$ & $1$ & $g$ & $-x_{j}$ & $-x_{j}\cdot g$ \\ 
    $x_{i}$ & $0$ & $0$ & $A_{ij}(1-g)$ & $A_{ij}(g-1)$ \\ 
    $x_{i}\cdot g$ & $0$ & $0$ & $A_{ij}(g-1)$ & $A_{ij}(1-g)$ \\ 
\end{tabular}\qquad
\begin{tabular}{c|cccc}
    $\beta_{\rightharpoonup}$& $1$ & $g$ & $x_{j}$ & $x_{j}\cdot g$\\ \hline
    $1$ & $1$ & $g$ & $x_{j}$ & $x_{j}\cdot g$ \\ 
    $g$ & $1$ & $g$ & $-x_{j}$ & $-x_{j}\cdot g$ \\ 
    $x_{i}$ & $0$ & $0$ & $A_{ij}(g-1)$ & $A_{ij}(1-g)$ \\
    $x_{i}\cdot g$ &  $0$ & $0$ & $A_{ij}(g-1)$ & $A_{ij}(1-g)$ \\ 
\end{tabular}
\]
\end{example}

\begin{example}\label{ex:Slq2}
Let $0\neq q\in\mathbb{C}$ and consider the $\mathbb{C}$-algebra $(H,\cdot,1)$ generated by $a,b,c,d$ modulo the relations 
\begin{align*}&a\cdot b = q^2 b\cdot a, &&b\cdot c-c\cdot b = (q^2-1)\big(a \cdot a-a\cdot d\big),  \\
&a\cdot c=q^{-2}c\cdot a,&& b\cdot d - d\cdot b=(q^2-1)a\cdot b, \\
& a\cdot d = d\cdot a,&& c\cdot d-d\cdot c = (1-q^2) c\cdot a,  
&\end{align*}
and $a\cdot d-q^{-2}c\cdot b= 1$. Define $\Delta$, $\epsilon$ and $S$ on $a,b,c,d$ as 
\[\Delta\begin{pmatrix}
    a &b\\
    c & d
\end{pmatrix} = \begin{pmatrix}
    a &b\\
    c & d
\end{pmatrix}\ot \begin{pmatrix}
    a &b\\
    c & d
\end{pmatrix},\qquad \epsilon \begin{pmatrix}
    a &b\\
    c & d
\end{pmatrix} = \begin{pmatrix}
    1&0\\
    0&1
\end{pmatrix}, \qquad S\begin{pmatrix}
    a &b\\
    c & d
\end{pmatrix} = \begin{pmatrix}
    q^{-2}d+(1-q^{-2})a & -q^{-2}b\\
    -q^{-2}c & a
\end{pmatrix}.
\]
We have that $(H,\cdot,1,\Delta,\epsilon,S)$ is the transmutation of the Hopf algebra $\mathrm{SL}_{q}(2)$, see \cite[6.3]{FS}. Moreover, $(H,\cdot,1,\Delta,\epsilon,S,\rightharpoonup)$ is a Yetter--Drinfeld post-Hopf algebra where $\rightharpoonup$ is given as in the following diagram.
\[
\begin{tabular}{c|cccc}
    $\rightharpoonup$& $ a $ & $ b $ & $ c $ & $ d $\\ \hline
    $ a $ & $  a  $ & $q^{-1}b  $ & $qc $ & $ d $ \\ 
    $ b $ & $ (1-q^{-2}) b  $ & $0 $ & $q(1-q^{-2})(d-a) $
 & $(1-q^2)b $ \\ 
    $ c $ & $ 0 $ & $ 0 $ & $0$ & $ 0$  \\ 
    $ d $ & $ a   $ & $  qb  $ & $q^{-1} c  $ & $ d $ \\ 
\end{tabular}
\]
Moreover, $\alpha_{\rightharpoonup}$ has convolution inverse $\beta_{\rightharpoonup}$ given as in the following diagram.
\[
\begin{tabular}{c|cccc}
    $\beta_{\rightharpoonup}$& $ a $ & $ b $ & $ c $ & $ d $\\ \hline
    $ a $ & $  a  $ & $qb  $ & $q^{-1}c $ & $ d $ \\ 
    $ b $ & $ q(q^{-2}-1) b  $ & $0 $ & $(1-q^{2})(d-a) $
 & $-q(1-q^2)b $ \\ 
    $ c $ & $ 0 $ & $ 0 $ & $0$ & $ 0$  \\ 
    $ d $ & $ a   $ & $  q^{-1}b  $ & $q c  $ & $ d $ \\ 
\end{tabular}
\]    
\end{example}
\begin{example}\label{ex:Suzuki}
Let $\Bbbk$ be an algebraically closed field of $\mathrm{char}(\Bbbk)\not=2$. Consider the algebra $(H,\cdot,1)$ generated by $a,b,c,d$ modulo the relations \begin{align*} &a\cdot a=d \cdot d, \quad c \cdot c = \alpha\beta^{-1} b\cdot b,\quad c\cdot b = b\cdot c,\quad  a\cdot d = d\cdot a,\\
&a\cdot b=b\cdot a=a\cdot c=c\cdot a=b\cdot d=d\cdot b=c\cdot d=d\cdot c=0,
\end{align*}
where $\alpha,\beta\in\Bbbk$. Define $\Delta$, $\epsilon$ and $S$ on $a,b,c,d$ as
\[\Delta\begin{pmatrix}
    a &b\\
    c & d
\end{pmatrix} = \begin{pmatrix}
    a &b\\
    c & d
\end{pmatrix}\ot \begin{pmatrix}
    a &b\\
    c & d
\end{pmatrix},\qquad \epsilon \begin{pmatrix}
    a &b\\
    c & d
\end{pmatrix} = \begin{pmatrix}
    1&0\\
    0&1
\end{pmatrix}, \qquad S\begin{pmatrix}
    a &b\\
    c & d
\end{pmatrix} = \begin{pmatrix}
    \alpha^{2}\beta^{2}(a\cdot d^{2}) & \alpha^{4}(b^{2}\cdot c)\\
    \alpha^{4}(b\cdot c^{2}) & \alpha^{2}\beta^{2}(d\cdot a^{2})
\end{pmatrix}.
\]
We have that $(H,\cdot,1,\Delta,\epsilon,S)$ is the transmutation of the Suzuki Hopf algebra $A_{1,2}^{\nu,\lambda}$, see \cite[6.4]{FS}. Moreover, $(H,\cdot,1,\Delta,\epsilon,S,\rightharpoonup)$ is a Yetter--Drinfeld post-Hopf algebra where $\rightharpoonup$ is given as in the following diagram, which coincides also with the diagram of $\beta_{\rightharpoonup}$.
\[
\begin{tabular}{c|cccc}
    $\rightharpoonup$& $ a $ & $ b $ & $ c $ & $ d $\\ \hline
    $ a $ & $  d $ & $ \alpha^{-1}\beta c $ & $\alpha\beta^{-1} b $ & $ a $ \\ 
    $ b $ & $ 0 $ & $0 $ & $0 $
 & $0$ \\ 
    $ c $ & $ 0 $ & $ 0 $ & $0$ & $ 0$  \\ 
    $ d $ & $d  $ & $  \alpha\beta^{-1}c  $ & $\alpha^{-1}\beta b $ & $ a $ \\ 
    \end{tabular}
\]
\end{example}

\begin{remark}
    Notice that, starting from a Hopf algebra $(H,\bullet,1,\Delta,\epsilon,T)$, one can obtain a Yetter--Drinfeld brace through the matched pair of actions $(\rightharpoonup,\leftharpoonup)$ given by the left adjoint action $a\rightharpoonup b:=a_{1}\bullet b\bullet T(a_{2})$ and the right trivial action $a\leftharpoonup b:=a\epsilon(b)$ as in \cite[Lemma 3.22]{FS}. Hence, one obtains a Yetter--Drinfeld post-Hopf algebra $(H,\cdot,1,\Delta,\epsilon,S,\rightharpoonup)$ where $\cdot$ and $S$ are defined by
\[
a\cdot b:=a_{1}\bullet T(a_{3})\bullet b\bullet T(T(a_{2})), \qquad S(a):=a_{1}\bullet T(a_{3})\bullet T(a_{2}).
\]
Recall that $\beta_{\rightharpoonup}a:b\mapsto(T(a)\rightharpoonup b)$. This can be a way to produce other examples of Yetter--Drinfeld post-Hopf algebras, which are not post-Hopf algebras in general. 
\end{remark}

\section{Yetter--Drinfeld relative Rota--Baxter operators}\label{sec:YetterDrinfeldRotaBaxter}
\noindent The notion of \textit{Rota--Baxter operator} on a cocommutative Hopf algebra $H$ was introduced in \cite{Go} as a morphism of coalgebras $R:H\to H$ satisfying
\[
R(x)R(y)=R(x_{1}\mathrm{ad}_{L}(R(x_{2})\otimes y))=R(x_{1}R(x_{2})yS(R(x_{3})))\quad \text{for all}\ x,y\in H,
\]
where $\mathrm{ad}_{L}$ denotes the left adjoint action $H\otimes H\to H$, $x\otimes y\mapsto x_{1}yS(x_{2})$, then generalised for arbitrary actions with the notion of \textit{relative Rota--Baxter operator} on a cocommutative Hopf algebra \cite[Definition 3.1]{YYT}. Moreover, in \cite[Theorem 3.4]{YYT} an adjunction between the categories of relative Rota--Baxter operators on cocommutative Hopf algebras and cocommutative post-Hopf algebras was provided. 

\begin{remark}
    In \cite[Definition 3.1]{YYT} $H$ and $K$ are cocommutative Hopf algebras and $K$ is also an object in $\mathrm{Bimon}(_{H}\mathfrak{M})$ (where the braiding considered for $_{H}\mathfrak{M}$ is given by the canonical flip, which corresponds to the quasitriangular structure $1_{H}\otimes1_{H}$ on $H$). But then, mimicking the proof of \cite[Proposition 3.10]{FS}, one obtains that $S_{K}$ is in $_{H}\mathfrak{M}$, hence $K$ is actually an object in $\mathrm{Hopf}(_{H}\mathfrak{M})$.
\end{remark}

We introduce a generalisation of the notion of relative Rota--Baxter operator, which we call \textit{Yetter--Drinfeld relative Rota--Baxter operator}. A subcategory of the category of bijective Yetter--Drinfeld relative Rota--Baxter operators will turn out to be equivalent to the category of Yetter--Drinfeld post-Hopf algebras.

\begin{definition}
    Let $H$ be a Hopf algebra and $K$ be an object in $\mathrm{Bimon}(^{H}_{H}\mathcal{YD})$, and let the action of $H$ on $K$ be denoted by $\rightharpoonup$. A coalgebra morphism $R:K\to H$ is called a \textbf{Yetter--Drinfeld relative 
    Rota--Baxter operator} on $H$ with respect to $(K,\rightharpoonup)$ if the following equalities hold, for all $a,b\in K$:
\begin{equation}\label{RotaBaxterequation}
R(a)\cdot R(b)=R(a_{1}\cdot(R(a_{2})\rightharpoonup b)),
\end{equation}
\begin{equation}\label{RotaBaxterequation3}
\begin{split}
S_{H}R(R(a_{1})\rightharpoonup b_{1})\cdot&R(a_{2})\cdot R(b_{2})\otimes 
R(R(a_{3})\rightharpoonup b_{3})=S_{H}R(R(a_{2})\rightharpoonup b_{2})\cdot R(a_{3})\cdot R(b_{3})\otimes 
R(R(a_{1})\rightharpoonup b_{1}).
\end{split}
\end{equation}

A \textit{morphism of Yetter--Drinfeld relative 
Rota--Baxter operators} from $R:K\to H$ to $R':K'\to H'$ is a pair of morphisms of algebras and coalgebras $(f:H\to H',g:K\to K')$ such that 
\begin{equation}\label{morphRotaBaxter}
fR=R'g, \qquad g\rightharpoonup\,=\,\rightharpoonup'(f\otimes g).
\end{equation}
The category of Yetter--Drinfeld relative 
Rota--Baxter operators and morphisms of Yetter--Drinfeld relative 
Rota--Baxter operators will be denoted by $\mathcal{YD}\mathrm{rRB}(\mathrm{Vec}_{\Bbbk})$. 
\end{definition}

\begin{remark}
    Since $R(1_{K})$ is a group-like element in $G(H)$, it has inverse (given by $S_{H}R(1_{K})$). Moreover, from \eqref{RotaBaxterequation} we have 
\[
R(1_{K})\cdot R(1_{K})=R(1_{K}\cdot(R(1_{K})\rightharpoonup1_{K}))=R(\epsilon(R(1_{K}))1_{K})=R(1_{K}),
\]
hence $R(1_{K})=1_{H}$. 
\end{remark}

\begin{remark}
Notice that, if $K$ is cocommutative, then \eqref{RotaBaxterequation3} is automatically satisfied. 
Moreover, if $R:K\to H$ is bijective, the cocommutativity of $K$ induces the cocommutativity of $H$ (and vice versa). We point out that a subcategory of bijective Yetter--Drinfeld relative Rota--Baxter operators will turn out to be equivalent to the category of Yetter--Drinfeld post-Hopf algebras and \eqref{RotaBaxterequation3} will correspond to \eqref{MP5}. 
\end{remark}

Even when $R$ is bijective, the morphism $R^{-1}S_{H}R$ is not an antipode for $K$ in general, since $R$ is not a morphism of algebras, but we can equip $K$ with an antipode in the following way.

\begin{lemma}\label{lem:antipodeK}
    Let $R:K\to H$ be a bijective Yetter--Drinfeld relative Rota--Baxter operator. Then, the morphism $S_{K}:K\to K$, $a\mapsto(R(a_{1})\rightharpoonup R^{-1}S_{H}R(a_{2}))$ satisfies $S_{K}(a_{1})\cdot a_{2}=\epsilon(a)1_{K}=a_{1}\cdot S_{K}(a_{2})$. As a consequence, $K$ is an object in $\mathrm{Hopf}(^{H}_{H}\mathcal{YD})$.
\end{lemma}

\begin{proof}
First we compute
\[
\begin{split}
R(a_{1}\cdot S_{K}(a_{2}))&=R\Big(a_{1}\cdot\big(R(a_{2})\rightharpoonup R^{-1}S_{H}R(a_{3})\big)\Big)\overset{\eqref{RotaBaxterequation}}{=}R(a_{1})\cdot RR^{-1}S_{H}R(a_{2})=R(a)_{1}\cdot S_{H}(R(a)_{2})\\&=\epsilon(R(a))1_{H}=\epsilon(a)1_{H},
\end{split}
\]
so that $a_{1}\cdot S_{K}(a_{2})=\epsilon(a)1_{K}$. Moreover, using that $\rightharpoonup$ is multiplicative, we also have 
\[
\begin{split}
R\Big(S_{H}R(a_{1})\rightharpoonup(S_{K}(a_{2})\cdot a_{3})\Big)&=R\Big(\big(S_{H}R(a_{1})_{1}\rightharpoonup S_{K}(a_{2})\big)\cdot\big(S_{H}R(a_{1})_{2}\rightharpoonup a_{3}\big)\Big)\\&=R\Big(\big(S_{H}R(a_{2})\rightharpoonup S_{K}(a_{3})\big)\cdot\big(S_{H}R(a_{1})\rightharpoonup a_{4}\big)\Big)\\&=R\Big(\big(S_{H}R(a_{2})\rightharpoonup(R(a_{3})\rightharpoonup R^{-1}S_{H}R(a_{4}))\big)\cdot\big(S_{H}R(a_{1})\rightharpoonup a_{5}\big)\Big)\\&=R\Big(R^{-1}S_{H}R(a_{2})\cdot(S_{H}R(a_{1})\rightharpoonup a_{3})\Big)\\&=R\Big(R^{-1}S_{H}R(a_{1})_{1}\cdot\big(R(R^{-1}S_{H}R(a_{1})_{2})\rightharpoonup a_{2}\big)\Big)\\&=RR^{-1}S_{H}R(a_{1})\cdot R(a_{2})=S_{H}(R(a)_{1})\cdot R(a)_{2}\\&=\epsilon(a)1_{H}
\end{split}
\]
so that $S_{H}R(a_{1})\rightharpoonup(S_{K}(a_{2})\cdot a_{3})=\epsilon(a)1_{K}$. Then, we obtain
\[
S_{K}(a_{1})\cdot a_{2}=R(a_{1})\rightharpoonup\big(S_{H}R(a_{2})\rightharpoonup(S_{K}(a_{3})\cdot a_{4})\big)=R(a_{1})\rightharpoonup\epsilon(a_{2})1_{K}=\epsilon(R(a))1_{K}=\epsilon(a)1_{K}
\]
and so $K$ is an object in $\mathrm{Hopf}(^{H}_{H}\mathcal{YD})$ by \cite[Proposition 3.10]{FS}.
\end{proof}

We denote by $\mathrm{B}\mathcal{YD}\mathrm{rRB}(\mathrm{Vec}_{\Bbbk})$ the category of bijective Yetter--Drinfeld relative Rota--Baxter operators.

\begin{remark}
    We have shown that every bijective Yetter--Drinfeld relative Rota--Baxter operator $R:K\to H$ is such that $K$ is in $\mathrm{Hopf}(^{H}_{H}\mathcal{YD})$. Clearly one can deduce the same for Yetter--Drinfeld 1-cocycles in the sense of \cite[Definition 4.1]{FS}, so that (20) in \cite[Definition 4.1]{FS} becomes redundant. We have that $R$ is a bijective Yetter--Drinfeld relative Rota--Baxter operator if and only if $R^{-1}$ is a Yetter--Drinfeld 1-cocycle in the sense of \cite[Definition 4.1]{FS}; 
this result generalises \cite[Proposition 3.18]{HLTL}.
    
    Moreover, denoting by $\Cc$ the full subcategory of $\mathcal{YD}\mathrm{r
    RB}(\mathrm{Vec}_{\Bbbk})$ whose objects are Yetter--Drinfeld relative 
    Rota--Baxter operators $R:K\to H$ such that $K$ is in $\mathrm{Hopf}(^{H}_{H}\mathcal{YD})$, the category 
    $\mathrm{B}\mathcal{YD}\mathrm{rRB}(\mathrm{Vec}_{\Bbbk})$
    is a subcategory of $\Cc$.
\end{remark}

A Yetter--Drinfeld relative 
Rota–Baxter operator naturally induces relative Rota–Baxter operators on the group $G(H)$ and on the Lie algebra $P(H)$. 

\begin{proposition}
    Let $R:K\to H$ be a Yetter--Drinfeld relative 
    Rota–Baxter operator on $H$ with respect to $(K,\rightharpoonup)$. Then, the following statements hold:
\begin{itemize}
    \item[1)] $R|_{G(K)}:G(K)\to G(H)$ is a relative Rota–Baxter operator on the group $G(H)$ with respect to $(G(K),\rightharpoonup)$,\medskip
    \item[2)] $R|_{P(K)}:P(K)\to P(H)$ is a relative Rota–Baxter operator on the Lie algebra $P(H)$ with respect to $(P(K),\rightharpoonup)$.
\end{itemize}
\end{proposition}

\begin{invisible}
\begin{proof}
    Since $R$ is a morphism of coalgebras, it follows that $R|_{G(K)}$ is a map from $G(K)$ to $G(H)$ and, also since $R(1_{K})=1_{H}$, it follows that $R|_{P(K)}$ is a map from $P(K)$ to $P(H)$. Moreover, for $a\in G(K)$, we have that $\alpha_{\rightharpoonup}R(a):G(K)\to G(K),x\mapsto(R(a)\rightharpoonup x)$ is an automorphism and, from \eqref{RotaBaxterequation}, we obtain $R(a)\cdot R(b)=R(a\cdot(R(a)\rightharpoonup b))$, so 1) holds true. Furthermore, for $a\in P(K)$, we have that $\alpha_{\rightharpoonup}R(a)$ is a morphism in $\mathrm{Der}(P(K))$. Moreover, for any $a,b\in P(K)$, we have 
\begin{equation}\label{eq:rb}
R(a)\cdot R(b)\overset{\eqref{RotaBaxterequation}}{=}R(1\cdot(R(a)\rightharpoonup b))+R(a\cdot(R(1)\rightharpoonup b))=R(R(a)\rightharpoonup b)+R(a\cdot b)
\end{equation}
and then
\[
\begin{split}
R(R(a)\rightharpoonup b)-R(R(b)\rightharpoonup a)+R([a,b])&\overset{\eqref{eq:rb}}{=}R(a)\cdot R(b)-R(a\cdot b)-R(b)\cdot R(a)+R(b\cdot a)+R(a\cdot b)-R(b\cdot a)\\&=[R(a),R(b)],
\end{split}
\]
hence also 2) is satisfied.
\end{proof}
\end{invisible}

The previous result generalises \cite[Theorem 3.2]{YYT} and it has the same proof, which only uses \eqref{RotaBaxterequation} and the fact that $R$ is a morphism of coalgebras. Hence, Yetter--Drinfeld relative Rota--Baxter operators are natural candidates in a not-cocomutative setting.

\subsection{Connection with Yetter--Drinfeld post-Hopf algebras}

A Yetter--Drinfeld post-Hopf algebra produces a Yetter--Drinfeld relative Rota--Baxter operator, as it is shown in the following result that generalises \cite[Proposition 3.3]{YYT} removing the hypothesis of cocommutativity. 

\begin{proposition}\label{prop:fromPHtorRB}
    Let $(H,\rightharpoonup)$ be a Yetter--Drinfeld post-Hopf algebra and $H_{\rightharpoonup}$ the subadjacent Hopf algebra defined in Proposition \ref{HharpHopf}. Then, the identity map $\mathrm{Id}_{H}:H\to H_{\rightharpoonup}$ is a (bijective) Yetter--Drinfeld relative Rota--Baxter operator on $H_{\rightharpoonup}$ with respect to $(H,\rightharpoonup)$. Moreover, if $g:(H,\rightharpoonup)\to(H',\rightharpoonup')$ is a morphism of Yetter--Drinfeld post-Hopf algebras, then $(g:H_{\rightharpoonup}\to H'_{\rightharpoonup'},g:H\to H')$ is a morphism of Yetter--Drinfeld relative Rota--Baxter operators from $\mathrm{Id}_{H}:H\to H_{\rightharpoonup}$ to $\mathrm{Id}_{H'}:H'\to H'_{\rightharpoonup'}$. This yields a functor $L:\mathcal{YD}\mathrm{PH}(\mathrm{Vec}_{\Bbbk})\to\mathcal{YD}\mathrm{rRB}(\mathrm{Vec}_{\Bbbk})$.
\end{proposition}

\begin{proof}
For any $a,b\in H$, we have 
\[
\mathrm{Id}_{H}(a)\bullet_{\rightharpoonup}\mathrm{Id}_{H}(b)=a\bullet_{\rightharpoonup}b\overset{\eqref{defbullet}}{=}a_{1}\cdot(a_{2}\rightharpoonup b)=\mathrm{Id}_{H}(a_{1}\cdot(\mathrm{Id}_{H}(a_{2})\rightharpoonup b)),
\]
so \eqref{RotaBaxterequation} holds true. Moreover, also
\eqref{RotaBaxterequation3} is satisfied:
\[
\begin{split}
S_{\rightharpoonup}\mathrm{Id}_{H}(\mathrm{Id}_{H}(a_{1})&\rightharpoonup b_{1})\bullet_{\rightharpoonup}\mathrm{Id}_{H}(a_{2})\bullet_{\rightharpoonup}\mathrm{Id}_{H}(b_{2})\otimes 
\mathrm{Id}_{H}
(\mathrm{Id}_{H}(a_{3})\rightharpoonup b_{3})=\\&\ =S_{\rightharpoonup}(a_{1}\rightharpoonup b_{1})\bullet_{\rightharpoonup}a_{2}\bullet_{\rightharpoonup}b_{2}\otimes 
(a_{3}\rightharpoonup b_{3})\\&\overset{\eqref{defleftharp}}{=}(a_{1}\leftharpoonup b_{1})\otimes 
(a_{2}\rightharpoonup b_{2})\\&\overset{\eqref{MP5}}{=}(a_{2}\leftharpoonup b_{2})\otimes 
(a_{1}\rightharpoonup b_{1})\\&\overset{\eqref{defleftharp}}{=}S_{\rightharpoonup}(a_{2}\rightharpoonup b_{2})\bullet_{\rightharpoonup}a_{3}\bullet_{\rightharpoonup}b_{3}\otimes 
(a_{1}\rightharpoonup b_{1})\\&\ =S_{\rightharpoonup}\mathrm{Id}_{H}(\mathrm{Id}_{H}(a_{2})\rightharpoonup b_{2})\bullet_{\rightharpoonup}\mathrm{Id}_{H}(a_{3})\bullet_{\rightharpoonup}\mathrm{Id}_{H}(b_{3})\otimes 
\mathrm{Id}_{H}
(\mathrm{Id}_{H}(a_{1})\rightharpoonup b_{1}).
\end{split}
\]
Thus, $\mathrm{Id}_{H}:H\to H_{\rightharpoonup}$ is a (bijective) Yetter--Drinfeld relative 
Rota--Baxter operator on $H_{\rightharpoonup}$ with respect to $(H,\rightharpoonup)$. 

Let $g:(H,\cdot,1,\Delta,\epsilon,S,\rightharpoonup)\to(H',\cdot',1',\Delta',\epsilon',S',\rightharpoonup')$ be a Yetter--Drinfeld post-Hopf algebra morphism, i.e. $g:H\to H'$ is a morphism of algebras and coalgebras which satisfies $g\rightharpoonup\,=\,\rightharpoonup'(g\otimes g)$. Clearly the pair $(g:H_{\rightharpoonup}\to H'_{\rightharpoonup'},g:H\to H')$ satisfies \eqref{morphRotaBaxter}. Moreover, we have
\[
g(x\bullet_{\rightharpoonup}y)\overset{\eqref{defbullet}}{=}g(x_{1}\cdot(x_{2}\rightharpoonup y))=g(x_{1})\cdot'g(x_{2}\rightharpoonup y)=g(x_{1})\cdot'(g(x_{2})\rightharpoonup' g(y))=g(x)\bullet_{\rightharpoonup'}g(y),
\]
so $g$ is also a morphism of algebras from $H_{\rightharpoonup}$ to $H'_{\rightharpoonup'}$ and the pair $(g:H_{\rightharpoonup}\to H'_{\rightharpoonup'},g:H\to H')$ is a morphism of Yetter--Drinfeld relative Rota--Baxter operators from $\mathrm{Id}_{H}:H\to H_{\rightharpoonup}$ to $\mathrm{Id}_{H'}:H'\to H'_{\rightharpoonup'}$. Clearly the assignment $L$ respects identities and compositions.
\end{proof}

\begin{remark}
The previous result can also be deduced in the following way. Given a Yetter--Drinfeld post-Hopf algebra $(H,\rightharpoonup)$ one obtains the Yetter--Drinfeld brace $(H,\cdot,\bullet_{\rightharpoonup},1,\Delta,\epsilon,S,S_{\rightharpoonup})$ as in Corollary \ref{cor:fromPHtoBr}. Then, from \cite[Theorem 4.3 (ii. to i.)]{FS}, one obtains the Yetter--Drinfeld 1-cocycle $\mathrm{Id}_{H}:H_{\rightharpoonup}\to H$, hence the Yetter--Drinfeld relative Rota--Baxter operator $\mathrm{Id}_{H}:H\to H_{\rightharpoonup}$.
\end{remark}

We can use Proposition \ref{prop:fromPHtorRB} to obtain examples of Yetter--Drinfeld relative Rota--Baxter operators.

\begin{example}
    In Examples \ref{ex:Sweedler}, \ref{ex:E(n)}, \ref{ex:Slq2} and \ref{ex:Suzuki} we obtained Yetter--Drinfeld post-Hopf algebras that coincide with the transmutation of the Sweedler's Hopf algebra $H_{4}$, the Hopf algebras $E(n)$, the Hopf algebra $\mathrm{SL}_{q}(2)$ and the Suzuki Hopf algebra $A_{1,2}^{\nu,\lambda}$, respectively. We denote them as \underline{$H_{4}$}, \underline{$E(n)$}, \underline{$\mathrm{SL}_q(2)$} and \underline{$A_{1,2}^{\nu,\lambda}$}. Therefore, using Proposition \ref{prop:fromPHtorRB}, we obtain that
\[    
    \mathrm{Id}:\underline{H_{4}}\to H_{4}, \quad \mathrm{Id}:\underline{E(n)}\to E(n), \quad \mathrm{Id}:\underline{\mathrm{SL}_q(2)}\to\mathrm{SL}_q(2),\quad \mathrm{Id}:\underline{A_{1,2}^{\nu,\lambda}}\to A_{1,2}^{\nu,\lambda}
\]
are all examples of Yetter--Drinfeld relative Rota--Baxter operators.
\end{example}

Let $\Dd$ be the full subcategory of $\mathrm{B}\mathcal{YD}\mathrm{rRB}(\mathrm{Vec}_{\Bbbk})$ whose objects are bijective Yetter--Drinfeld relative Rota--Baxter operators $R:K\to H$ such that $K$ has coaction given by $(\mathrm{Id}_{H}\otimes R^{-1})\mathrm{Ad}^{H}_{L}R$, where $\mathrm{Ad}^{H}_{L}$ denotes the left adjoint coaction on $H$ (equivalently, $R:K\to H$ is left $H$-colinear by considering $H$ equipped with $\mathrm{Ad}^{H}_{L}$).

\begin{remark}
If $H$ is cocommutative then $\mathrm{Ad}_{L}^{H}$ is the trivial $H$-coaction on $H$ and so the previous $H$-coaction on $K$ is trivial. Hence, in this case, $\sigma^{\mathcal{YD}}_{K,K}$ becomes the canonical flip $\tau_{K,K}$ and $K$ becomes a standard (cocommutative) Hopf algebra (and an object in $\mathrm{Hopf}(_{H}\mathfrak{M})$). Thus, the subcategory of $\Dd$ whose objects $R:K\to H$ are such that $H$ is cocommutative coincides with the category of bijective relative Rota--Baxter operators on cocommutative Hopf algebras, see \cite[Definition 3.1]{YYT}. 
\end{remark}



As it is shown in the following result, we can build a functor $M:\Dd\to\mathcal{YD}\mathrm{PH}(\mathrm{Vec}_{\Bbbk})$. 

\begin{proposition}\label{prop:PHonH}
    Let $R:K\to H$ be a bijective Yetter--Drinfeld relative Rota--Baxter operator on the Hopf algebra $(H,\cdot,1,\Delta,\epsilon,S)$ with respect to $(K,\rightharpoonup)$, where $K$ is in $^{H}\mm$ with coaction given by $(\mathrm{Id}_{H}\otimes R^{-1})\mathrm{Ad}^{H}_{L}R$. Then, $(H,\cdot_{R},1,\Delta,\epsilon,S_{R},\rightharpoonup_{R})$ is a Yetter--Drinfeld post-Hopf algebra where $\cdot_{R}$ and $S_{R}$ are defined, for all $a,b\in H$, by
\begin{equation}\label{productandantipode}
a\cdot_{R}b=R(R^{-1}(a)\cdot_{K}R^{-1}(b)), \qquad S_{R}(a):=R(a_{1}\rightharpoonup R^{-1}S(a_{2}))
\end{equation}
and $\rightharpoonup_{R}$ is defined by
\[
a\rightharpoonup_{R}b:=R(a\rightharpoonup R^{-1}(b)).
\]
\end{proposition}

\begin{proof}
    Given a bijective Yetter--Drinfeld relative Rota--Baxter operator $R:K\to H$, we obtain a Yetter--Drinfeld 1-cocycle $R^{-1}:H\to K$. Thus, we can apply \cite[Theorem 4.3 (i. to ii.)]{FS} obtaining a Yetter--Drinfeld brace $(H,\cdot_{R},\cdot,1,\Delta,\epsilon,S_{R},S)$, where $\cdot_{R}$ and $S_{R}$ are defined as in \eqref{productandantipode}. Then, by Proposition \ref{prop:fromYDbracetoYDpostHopf} we have that $(H,\cdot_{R},1,\Delta,\epsilon,S_{R},\rightharpoonup_{R})$ is a Yetter--Drinfeld post-Hopf algebra where $\rightharpoonup_{R}$ is defined by
\[
\begin{split}
a\rightharpoonup_{R}b&:=S_{R}(a_{1})\cdot_{R}(a_{2}\cdot b)=R(a_{1}\rightharpoonup R^{-1}S(a_{2}))\cdot_{R}(a_{3}\cdot b)\\&\ =R\Big(R^{-1}R(a_{1}\rightharpoonup R^{-1}S(a_{2}))\cdot_{K}R^{-1}(a_{3}\cdot b)\Big)
\\&\ \overset{(!)}{=}R\Big(\big(a_{1}\rightharpoonup R^{-1}S(a_{2})\big)\cdot_{K}R^{-1}(a_{3})\cdot_{K}\big(a_{4}\rightharpoonup R^{-1}(b)\big)\Big)\\&\ =R\Big(S_{K}(R^{-1}(a_{1}))\cdot_{K}R^{-1}(a_{2})\cdot_{K}\big(a_{3}\rightharpoonup R^{-1}(b)\big)\Big)
\\&\ =R\Big(\epsilon(a_{1})1_{K}\cdot_{K}\big(a_{2}\rightharpoonup R^{-1}(b)\big)\Big)
\\&\ =R\big(a\rightharpoonup R^{-1}(b)\big),
\end{split}
\]
where the equality marked with $(!)$ follows using that $R^{-1}$ is a Yetter--Drinfeld 1-cocycle. 
\end{proof}

\begin{remark}\label{remark}
    Since we know that $(H,\cdot_{R},\cdot,1,\Delta,\epsilon,S_{R},S)$ is a Yetter--Drinfeld brace we immediately have that $H^{\cdot}:=(H,\cdot,1,\Delta,\epsilon,S)$ is the subadjacent Hopf algebra of $(H,\cdot_{R},1,\Delta,\epsilon,S_{R})$, so that $(H,\cdot_{R},1,\Delta,\epsilon,S_{R})$ is an object in $\mathrm{Hopf}(^{H^{\cdot}}_{H^{\cdot}}\mathcal{YD})$.
\end{remark}



Let us observe that the functor $L$ given in Proposition \ref{prop:fromPHtorRB} goes into the category $\Dd$. 
In analogy with \cite[Theorem 4.3]{FS}, using functors $L$ and $M$, we obtain the following result.

\begin{corollary}\label{cor:eqPHRB}
   The category $\mathcal{YD}\mathrm{PH}(\mathrm{Vec}_{\Bbbk})$ and the subcategory $\Dd$ of $\mathrm{B}\mathcal{YD}\mathrm{rRB}(\mathrm{Vec}_{\Bbbk})$ 
   are equivalent. 
\end{corollary}

\begin{proof}
    Let $R:K\to H$ be an object in $\mathcal{D}$. Then, by Remark \ref{remark},
    $LM((R:K\to H))$ is the Yetter--Drinfeld relative Rota--Baxter operator $\mathrm{Id}_{H}:(H,\cdot_{R},1,\Delta,\epsilon,S_{R},\rightharpoonup_{R})\to(H,\cdot,1,\Delta,\epsilon,S)$. Since $\mathrm{Id}_{H}(a)\rightharpoonup_{R}R(b)=R(a\rightharpoonup R^{-1}R(b))=R(a\rightharpoonup b)$ and $R$ is a morphism of algebras with respect to the new algebra structure $(H,\cdot_{R},1)$, we have that $(\mathrm{Id}_{H},R)$ is an isomorphism of Yetter--Drinfeld relative Rota--Baxter operators between $R:K\to H$ and $LM((R:K\to H))$ (and it is clearly natural). Hence $LM\cong\mathrm{Id}$. On the other hand 
\[    
ML((H,\cdot,1,\Delta,\epsilon,S,\rightharpoonup))=M\Big(\mathrm{Id}_{H}:(H,\cdot,1,\Delta,\epsilon,S,\rightharpoonup)\to(H,\bullet_{\rightharpoonup},1,\Delta,\epsilon,S_{\rightharpoonup})\Big)=(H,\cdot,1,\Delta,\epsilon,S,\rightharpoonup),
\]
since $a_{1}\rightharpoonup S_{\rightharpoonup}(a_{2})=\alpha_{\rightharpoonup}a_{1}(\beta_{\rightharpoonup}a_{2}(S(a_{3})))=S(a)$. Thus $ML=\mathrm{Id}$.
\end{proof}

    
Restricting the previous equivalence, we obtain an equivalence between cocommutative post-Hopf algebras and bijective relative Rota--Baxter operators on cocommutative Hopf algebras.

\begin{remark}
If we restrict the functor $L$ to cocommutative Yetter--Drinfeld post-Hopf algebras we recover the functor given in \cite[Proposition 3.3]{YYT} from the category of cocommutative post-Hopf algebras to the category of (bijective) relative Rota--Baxter operators on cocommutative Hopf algebras. Moreover, in \cite[Theorem 3.4]{YYT} it is proven that this functor has a right adjoint. We want to recover this result from a different equivalence between the category of Yetter--Drinfeld post-Hopf algebras and the subcategory $\Dd$ of the category of bijective Yetter--Drinfeld relative Rota--Baxter operators.
\end{remark}

Given a bijective Yetter--Drinfeld relative Rota--Baxter operator $R:K\to H$ in $\mathcal{D}$, we want to induce a structure of Yetter--Drinfeld post-Hopf algebra on $K$.

\begin{proposition}\label{prop:secondequivalence}
Let $R:K\to H$ be an object in $\Dd$.
Then, $(K,\rightharpoonup_{R})$ is a Yetter--Drinfeld post-Hopf algebra where $\rightharpoonup_{R}$ is given by
\begin{equation}
a\rightharpoonup_{R}b:=R(a)\rightharpoonup b, \qquad \text{for all}\ a,b\in K.
\end{equation}
Let $R:K\to H$ and $R':K'\to H'$ be 
objects in $\mathcal{D}$ and $(f:H\to H',g:K\to K')$ a morphism between them. Then, $g$ is a morphism of Yetter--Drinfeld post-Hopf algebras from $(K,\rightharpoonup_{R})$ to $(K',\rightharpoonup'_{R'})$. This yields a functor $\mathsf{R}:\Dd\to\mathcal{YD}\mathrm{PH}(\mathrm{Vec}_{\Bbbk})$.
\end{proposition}

\begin{proof}
    Let us observe that $\rightharpoonup_{R}:=\rightharpoonup(R\otimes\mathrm{Id}_{K})$. Thus, since $R$ and $\rightharpoonup$ are morphisms of coalgebras, also $\rightharpoonup_{R}$ is a morphism of coalgebras. Moreover, we already know that $(K,\cdot,1)$ is an algebra and $(K,\Delta,\epsilon)$ is a coalgebra and $S_{K}:K\to K$ satisfies $a_{1}\cdot S_{K}(a_{2})=\epsilon(a)1_{K}=S_{K}(a_{1})\cdot a_{2}$, for all $k\in K$. Moreover, since $\rightharpoonup$ satisfies \eqref{dotlinear}, the same holds for $\rightharpoonup_{R}$:
\[
\begin{split}
a\rightharpoonup_{R}(b\cdot c)&=R(a)\rightharpoonup(b\cdot c)\overset{\eqref{dotlinear}}{=}
(R(a_{1})\rightharpoonup b)\cdot(R(a_{2})\rightharpoonup c)=(a_{1}\rightharpoonup_{R}b)\cdot(a_{2}\rightharpoonup_{R}c).
\end{split}
\]
We already know that $\rightharpoonup$ is an $H$-action on $K$, so also \eqref{deformedassociativityright} is satisfied:
\[
\begin{split}
    a\rightharpoonup_{R}(b\rightharpoonup_{R}c)&=R(a)\rightharpoonup(R(b)\rightharpoonup c)=(R(a)\cdot R(b))\rightharpoonup c\overset{\eqref{RotaBaxterequation}}{=}R(a_{1}\cdot(R(a_{2})\rightharpoonup b))\rightharpoonup c\\&=
(a_{1}\cdot(a_{2}\rightharpoonup_{R}b))\rightharpoonup_{R}c.
\end{split}
\]
For all $a,b\in K$, define $(\beta_{\rightharpoonup_{R}}a)(b):=S_{H}(R(a))\rightharpoonup b$, so that
\[
\begin{split}
\beta_{\rightharpoonup_{R}}a_{1}(\alpha_{\rightharpoonup_{R}}a_{2}(b))&=
S_{H}R(a_{1})\rightharpoonup(R(a_{2})\rightharpoonup b)=
(S_{H}(R(a_{1}))\cdot R(a_{2}))\rightharpoonup b=\epsilon(R(a))1\rightharpoonup b=\epsilon(a)b,
\end{split}
\]
hence $(\beta_{\rightharpoonup_{R}}a_{1})\circ(\alpha_{\rightharpoonup_{R}}a_{2})=\epsilon(a)\mathrm{Id}_{K}$ for all $a\in K$. Similarly, $(\alpha_{\rightharpoonup_{R}}a_{1})\circ(\beta_{\rightharpoonup_{R}}a_{2})=\epsilon(a)\mathrm{Id}_{K}$, for all $a\in K$. We already know that $\epsilon(a\cdot b)=\epsilon(a)\epsilon(b)$ for all $a,b\in K$, $\epsilon(1_{K})=1_{\Bbbk}$ and $\Delta(1_{K})=1_{K}\otimes1_{K}$. Moreover, since $K$ is in $\mathrm{Bimon}(^{H}_{H}\mathcal{YD})$ and it is in $^{H}\mm$ with $(\mathrm{Id}_{H}\otimes R^{-1})\mathrm{Ad}^{H}_{L}R$, we have
\[
\sigma^{\mathcal{YD}}_{K,K}(a\otimes b):=
(R(a_{1})\cdot S_{H}(R(a_{3}))\rightharpoonup b)\otimes a_{2},
\]
so that
\[
\begin{split}
\Delta(a\cdot b)&=(m\otimes m)(\mathrm{Id}_{K}\otimes\sigma^{\mathcal{YD}}_{K,K}\otimes\mathrm{Id}_{K})(\Delta\otimes\Delta)(a\otimes b)\\&=a_{1}\cdot\big(R(a_{2})\cdot S_{H}(R(a_{4}))\rightharpoonup b_{1}\big)\otimes(a_{3}\cdot b_{2})\\&=a_{1}\cdot\big(R(a_{2})\rightharpoonup(S_{H}(R(a_{4}))\rightharpoonup b_{1})\big)\otimes(a_{3}\cdot b_{2})\\&=a_{1}\cdot\big(\alpha_{\rightharpoonup_{R}}a_{2}(\beta_{\rightharpoonup_{R}}a_{4}(b_{1}))\big)\otimes(a_{3}\cdot b_{2}),
\end{split},
\]
hence \eqref{comp.Deltadot} holds true. Furthermore, \eqref{defbullet} and  \eqref{defantipode} become
\[
a\bullet_{\rightharpoonup_{R}}b:=a_{1}\cdot(a_{2}\rightharpoonup_{R}b)=a_{1}\cdot(R(a_{2})\rightharpoonup b)\overset{\eqref{RotaBaxterequation}}{=}R^{-1}(R(a)\cdot R(b))
\]
and
\[
S_{\rightharpoonup_{R}}(a):=\beta_{\rightharpoonup_{R}}a_{1}(S_{K}(a_{2}))=S_{H}(R(a_{1}))\rightharpoonup S_{K}(a_{2}).
\]
But, by Lemma \ref{lem:antipodeK}, we know that $S_{K}(a)=R(a_{1})\rightharpoonup R^{-1}S_{H}R(a_{2})$ and then
\[
\begin{split}
S_{\rightharpoonup_{R}}(a)&=S_{H}(R(a_{1}))\rightharpoonup\big(R(a_{2})\rightharpoonup R^{-1}S_{H}R(a_{3})\big)=\big(S_{H}(R(a_{1}))\cdot R(a_{2})\big)\rightharpoonup R^{-1}S_{H}R(a_{3})=R^{-1}S_{H}R(a).
\end{split}
\]
Thus, $S_{\rightharpoonup_{R}}=R^{-1}S_{H}R$ and, since $R$ and $R^{-1}$ are comultiplicative and $S_{H}$ is anti-comultiplicative, $S_{\rightharpoonup_{R}}$ is anti-comultiplicative. Finally, \eqref{defleftharp} becomes
\[
\begin{split}
a\leftharpoonup_{R} b&:=S_{\rightharpoonup_{R}}(a_{1}\rightharpoonup_{R}b_{1})\bullet_{\rightharpoonup_{R}}a_{2}\bullet_{\rightharpoonup_{R}}b_{2}\\&\ =
R^{-1}\Big(RS_{\rightharpoonup_{R}}(a_{1}\rightharpoonup_{R}b_{1})\cdot R(a_{2})\cdot R(b_{2})\Big)\\&
\ =R^{-1}\Big(S_{H}R(R(a_{1})\rightharpoonup b_{1})\cdot R(a_{2})\cdot R(b_{2})\Big).
\end{split}
\]
So 
we can compute
\[
\begin{split}
(a_{1}\leftharpoonup_{R} b_{1})\otimes R(a_{2}\rightharpoonup_{R}b_{2}&)=R^{-1}\Big(S_{H}R(R(a_{1})\rightharpoonup b_{1})\cdot R(a_{2})\cdot R(b_{2})\Big)\otimes R\big(R(a_{3})\rightharpoonup b_{3}\big)
\\&\overset{\eqref{RotaBaxterequation3}}{=}R^{-1}\Big(S_{H}R(R(a_{2})\rightharpoonup b_{2})\cdot R(a_{3})\cdot R(b_{3})\Big)\otimes R\big(R(a_{1})\rightharpoonup b_{1}\big)\\&\ =(a_{2}\leftharpoonup_{R} b_{2})\otimes R(a_{1}\rightharpoonup_{R}b_{1}).
\end{split}
\]
Thus, since $R$ is injective, also \eqref{MP5} is satisfied and $(K,\rightharpoonup_{R})$ is a Yetter--Drinfeld post-Hopf algebra. 

Let $(f:H\to H',g:K\to K')$ be a morphism of Yetter--Drinfeld relative Rota–Baxter operators from $R:K\to H$ to $R':K'\to H'$. In particular, $g$ is a morphism of algebras and coalgebras; moreover it satisfies
\[
g(a\rightharpoonup_{R}b)=g(R(a)\rightharpoonup b)=f(R(a))\rightharpoonup' g(b)=R'(g(a))\rightharpoonup' g(b)=g(a)\rightharpoonup'_{R'}g(b),
\]
so $g$ is a morphism of Yetter--Drinfeld post-Hopf algebras from $(K,\rightharpoonup_{R})$ to $(K',\rightharpoonup'_{R'})$. Clearly, the assignment $\mathsf{R}$ respects identities and compositions.
\end{proof}

Hence, using the functors $L$ and $\mathsf{R}$, we obtain another equivalence between $\mathcal{YD}\mathrm{PH}(\mathrm{Vec}_{\Bbbk})$ and the subcategory $\Dd$ of $\mathrm{B}\mathcal{YD}\mathrm{rRB}(\mathrm{Vec}_{\Bbbk})$ which again restricts to an equivalence between cocommutative post-Hopf algebras and bijective relative Rota--Baxter operators on cocommutative Hopf algebras.

\begin{corollary}
    The categories $\mathcal{YD}\mathrm{PH}(\mathrm{Vec}_{\Bbbk})$ and $\mathcal{D}$ are equivalent.
\end{corollary}

\begin{proof}
   Given a Yetter--Drinfeld post-Hopf algebra $(H,\rightharpoonup)$, we have $\mathsf{R}L((H,\rightharpoonup))=\mathsf{R}\big(\mathrm{Id}_{H}:H\to H_{\rightharpoonup}\big)=(H,\rightharpoonup)$, so $\mathsf{R}L=\mathrm{Id}$. On the other hand, given $R:K\to H$ in $\mathcal{D}$, we have
\[
L\mathsf{R}((R:K\to H))=L((K,\rightharpoonup_{R}))=\big(\mathrm{Id}_{K}:(K,\rightharpoonup_{R})\to K_{R}\big).
\]
The pair $(R^{-1}:H\to K_{R},\mathrm{Id}_{K})$ clearly satisfies \eqref{morphRotaBaxter}: indeed $R^{-1}(a)\rightharpoonup_{R}\mathrm{Id}_{K}(b)=RR^{-1}(a)\rightharpoonup b=\mathrm{Id}_{K}(a\rightharpoonup b)$. Moreover, $R^{-1}$ is a morphism of coalgebras and 
\[
R^{-1}(a)\bullet_{\rightharpoonup_{R}}R^{-1}(b)=R^{-1}(a_{1})\cdot(a_{2}\rightharpoonup R^{-1}(b))=R^{-1}(a\cdot b),
\]
where the last equality follows since $R^{-1}$ is a Yetter--Drinfeld 1-cocycle. Thus, $(R^{-1},\mathrm{Id}_{K})$ is an isomorphism (which is clearly natural) and so $L\mathsf{R}\cong\mathrm{Id}$.
\end{proof}

\begin{remark}
Notice that the bijectivity of $R:K\to H$ is not used in the first part of Proposition \ref{prop:secondequivalence}. Hence one can try to obtain a similar result considering the category $\mathcal{YD}\mathrm{r
RB}(\mathrm{Vec}_{\Bbbk})$ of Yetter--Drinfeld relative 
Rota--Baxter operators and its subcategory $\Cc$ whose objects $R:K\to H$ are such that $K$ is in $\mathrm{Hopf}(^{H}_{H}\mathcal{YD})$. 
\end{remark}

Define the subcategory $\Cc'$ of $\Cc$ whose objects are injective Yetter--Drinfeld relative 
Rota--Baxter operators $R:K\to H$ such that $K$ has $H$-coaction given by $a\mapsto R(a_{1})S_{H}(R(a_{3}))\otimes a_{2}$. 

\begin{remark}
The difference between the categories $\Cc'$ and $\Dd$ is simply that an object $R:K\to H$ in $\Cc'$ is not necessarily invertible; the category $\Dd$ is a subcategory of $\Cc'$. Moreover, the subcategory of $\Cc'$ whose objects $R:K\to H$ are such that $H$ is cocommutative (so also $K$ is cocommutative since $R$ is injective) is the category of injective relative Rota--Baxter operators on cocommutative Hopf algebras, see \cite[Definition 3.1]{YYT}. 
\end{remark}

Notice that the functor $L$ of Proposition \ref{prop:fromPHtorRB} goes into the category $\Cc'$. We can then show the following result.

\begin{proposition}
    Let $R:K\to H$ be an object in $\Cc'$. Then, $(K,\rightharpoonup_{R})$ is a Yetter--Drinfeld post-Hopf algebra where $\rightharpoonup_{R}$ is defined by $a\rightharpoonup_{R}b:=R(a)\rightharpoonup b$, for all $a,b\in K$. Consider also the same assignment on morphisms of the functor $\mathsf{R}$ given in Proposition \ref{prop:secondequivalence}. This yields a functor $\mathsf{R}':\Cc'\to\mathcal{YD}\mathrm{PH}(\mathrm{Vec}_{\Bbbk})$, which coincides with $\mathsf{R}$ when it is restricted to $\Dd$. Moreover, the functor $\mathsf{R}'$ is a right adjoint of the functor $L$.
\end{proposition}

\begin{proof}
The first part of the proof of Proposition \ref{prop:secondequivalence} remains true in this setting. Then, we have
\[
a\bullet_{\rightharpoonup_{R}}b:=a_{1}\cdot(R(a_{2})\rightharpoonup b),\qquad S_{\rightharpoonup_{R}}(a):=S_{H}(R(a_{1}))\rightharpoonup S_{K}(a_{2}).
\]
In order to obtain that $(K,\rightharpoonup_{R})$ is a Yetter--Drinfeld post-Hopf algebra it only remains to prove that $S_{\rightharpoonup_{R}}$ is anti-comultiplicative and the pair $(\rightharpoonup_{R},\leftharpoonup_{R})$ satisfies \eqref{MP5}, where 
\[
a\leftharpoonup_{R}b:=S_{\rightharpoonup_{R}}(a_{1}\rightharpoonup_{R}b_{1})\bullet_{\rightharpoonup_{R}}a_{2}\bullet_{\rightharpoonup_{R}}b_{2}.
\]
Observe that, in term of $\bullet_{\rightharpoonup_{R}}$, \eqref{RotaBaxterequation} reads $R(a)\cdot R(b)=R(a\bullet_{\rightharpoonup_{R}}b)$ and then \eqref{RotaBaxterequation3} becomes:
\begin{equation}\label{auxiliaryeq}
S_{H}R(a_{1}\rightharpoonup_{R} b_{1})\cdot R(a_{2}\bullet_{\rightharpoonup_{R}}b_{2})\otimes 
R(a_{3}\rightharpoonup_{R} b_{3})=S_{H}R(a_{2}\rightharpoonup_{R} b_{2})\cdot R(a_{3}\bullet_{\rightharpoonup_{R}}b_{3})\otimes 
R(a_{1}\rightharpoonup_{R} b_{1}).
\end{equation}
Let us observe that  
\begin{equation}\label{propPhi}
R(a_{1})\rightharpoonup S_{\rightharpoonup_{R}}(a_{2})=R(a_{1})\rightharpoonup\big(S_{H}R(a_{2})\rightharpoonup S_{K}(a_{3})\big)=\big(R(a_{1})\cdot S_{H}(R(a_{2}))\big)\rightharpoonup S_{K}(a_{3})=S_{K}(a)
\end{equation}
and so
\begin{equation}\label{propPhi2}
    R(a_{1})\cdot R(S_{\rightharpoonup_{R}}(a_{2}))\overset{\eqref{RotaBaxterequation}}{=}R\big(a_{1}\cdot(R(a_{2})\rightharpoonup S_{\rightharpoonup_{R}}(a_{3}))\big)=R(a_{1}\cdot S_{K}(a_{2}))=\epsilon(a)1_{H}.
\end{equation}
Thus, we obtain
\begin{equation}\label{propPhi3}
R(S_{\rightharpoonup_{R}}(a))=S_{H}(R(a_{1}))\cdot R(a_{2})\cdot R(S_{\rightharpoonup_{R}}(a_{3}))\overset{\eqref{propPhi2}}{=}S_{H}(R(a)),
\end{equation}
i.e. $RS_{\rightharpoonup_{R}}=S_{H}R$. Therefore, since $R$ is injective and comultiplicative and $S_{H}$ is anti-comultiplicative, $S_{\rightharpoonup_{R}}$ is anti-comultiplicative. Moreover, \eqref{auxiliaryeq} becomes
\[
R\big(S_{\rightharpoonup_{R}}(a_{1}\rightharpoonup_{R} b_{1})\bullet_{\rightharpoonup_{R}}a_{2}\bullet_{\rightharpoonup_{R}}b_{2}\big)\otimes 
R(a_{3}\rightharpoonup_{R} b_{3})=R\big(S_{\rightharpoonup_{R}}(a_{2}\rightharpoonup_{R} b_{2})\bullet_{\rightharpoonup_{R}}a_{3}\bullet_{\rightharpoonup_{R}}b_{3}\big)\otimes 
R(a_{1}\rightharpoonup_{R} b_{1}),
\]
which is exactly 
\[
R(a_{1}\leftharpoonup_{R} b_{1})\otimes R(a_{2}\rightharpoonup_{R} b_{2})=R(a_{2}\leftharpoonup_{R}b_{2})\otimes R(a_{1}\rightharpoonup_{R}b_{1}).
\]
Since $R$ is injective we get that the pair $(\rightharpoonup_{R},\leftharpoonup_{R})$ satisfies \eqref{MP5}, hence $(K,\rightharpoonup_{R})$ is a Yetter--Drinfeld post-Hopf algebra. As in the proof of Proposition \ref{prop:secondequivalence} one concludes that $\mathsf{R}'$ is a functor.

Finally, we just have to prove that, for every $R:K\to H$ in $\Cc'
$ and $(H',\rightharpoonup')$ in $\mathcal{YD}\mathrm{PH}(\mathrm{Vec}_{\Bbbk})$, there is a natural bijection between $\mathrm{Hom}(L(H',\rightharpoonup'),R:K\to H)$ and $\mathrm{Hom}((H',\rightharpoonup'),\mathsf{R}'(R:K\to H))$, i.e. between $\mathrm{Hom}(\mathrm{Id}_{H'}:H'\to H'_{\rightharpoonup'},R:K\to H)$ and $\mathrm{Hom}((H',\rightharpoonup'),(K,\rightharpoonup_{R}))$. Given a morphism $(f:H'_{\rightharpoonup'}\to H,g:H'\to K)$ in $\mathrm{Hom}(\mathrm{Id}_{H'}:H'\to H'_{\rightharpoonup'},R:K\to H)$ we take $g$ as a morphism in $\mathrm{Hom}((H',\rightharpoonup'),(K,\rightharpoonup_{R}))$. Indeed, $g$ is a morphism of algebras and coalgebras; moreover, since $f=f\mathrm{Id}_{H'}=Rg$, we have
\[
g(a\rightharpoonup' b)=f(a)\rightharpoonup_{K}g(b)=R(g(a))\rightharpoonup_{K}g(b)=g(a)\rightharpoonup_{R}g(b).
\]
On the other hand, given $g:(H',\rightharpoonup')\to(K,\rightharpoonup_{R})$ in $\mathrm{Hom}((H',\rightharpoonup'),(K,\rightharpoonup_{R}))$ we take $(Rg:H'_{\rightharpoonup'}\to H,g:H'\to K)$ as a morphism in $\mathrm{Hom}(\mathrm{Id}_{H'}:H'\to H'_{\rightharpoonup'},R:K\to H)$. Indeed, $g:H'\to K$ is a morphism of algebras and coalgebras, so $Rg$ is a morphism of coalgebras. Moreover, $Rg$ is also a morphism of algebras:
\[
\begin{split}
Rg(a\bullet_{\rightharpoonup'}b)&=Rg(a_{1}\cdot(a_{2}\rightharpoonup' b))=R\big(g(a_{1})\cdot g(a_{2}\rightharpoonup'b)\big)=R\big(g(a_{1})\cdot(g(a_{2})\rightharpoonup_{R}g(b))\big)\\&=R\big(g(a_{1})\cdot(Rg(a_{2})\rightharpoonup g(b))\big)=Rg(a)\cdot Rg(b)
\end{split}
\]
and \eqref{morphRotaBaxter} is trivially satisfied.
Clearly the two assignments are inverse to each other and the bijection is natural.
\end{proof}

In the cocommutative setting we obtain an adjunction between cocommutative post-Hopf algebras and injective relative Rota--Baxter operators on cocommutative Hopf algebras. This is exactly the adjunction given in \cite[Theorem 3.4]{YYT}, which is given for relative Rota--Baxter operators on cocommutative Hopf algebras not necessarily injective. These can be seen as objects $R:K\to H$ in $\Cc$ where $K$ and $H$ are cocommutative and the $H$-coaction on $K$ is given by $a\mapsto R(a_{1})S_{H}(R(a_{3}))\otimes a_{2}$: an example is given by the extension of a relative Rota–Baxter operator on a Lie algebra to the respective universal enveloping algebras, as it is proved in \cite[Theorem 3.6]{YYT}.


\begin{remark}
Assume the hypotheses of the previous result are satisfied. By Proposition \ref{HharpHopf} we have that $(K,\bullet_{\rightharpoonup_{R}},1,\Delta,\epsilon,S_{\rightharpoonup_{R}})$ is a Hopf algebra, where 
\[
a\bullet_{\rightharpoonup_{R}}b:=a_{1}\cdot(R(a_{2})\rightharpoonup b), \qquad S_{\rightharpoonup_{R}}(a):=S_{H}R(a_{1})\rightharpoonup S_{K}(a_{2})
,
\]
which we call the \textit{descendent Hopf algebra} $K_{R}$ in analogy with \cite[Corollary 3.5]{YYT}. We automatically have that $(K,\cdot,\bullet_{\rightharpoonup_{R}},1,\Delta,\epsilon,S_{K},S_{\rightharpoonup_{R}})$ is a Yetter--Drinfeld brace by Corollary \ref{cor:fromPHtoBr}; this result generalises \cite[Proposition 3.16]{HLTL}. Moreover, using Corollary \ref{cor:fromPHtoMP}, one obtains a matched pair of actions $(K_{R},\rightharpoonup_{R},\leftharpoonup_{R})$, where $a\leftharpoonup_{R}b:=S_{\rightharpoonup_{R}}(a_{1}\rightharpoonup_{R}b_{1})\bullet_{\rightharpoonup_{R}}a_{2}\bullet_{\rightharpoonup_{R}}b_{2}$; this result generalises \cite[Corollary 4.10]{YYT}.
\end{remark}




\bigskip

\noindent\textbf{Acknowledgements}. 
The author would like to thank A. Ardizzoni for useful comments and D. Ferri for his interest in this paper and a nice discussion regarding possible future developments. 
Furthermore, he would like to express his gratitude to the referee for meaningful suggestions. This paper was written while the author was member of the “National Group for Algebraic and Geometric Structures and their Applications” (GNSAGA-INdAM). He was also partially supported by the project funded by the European Union -NextGenerationEU under NRRP, Mission 4 Component 2 CUP D53D23005960006 - Call PRIN 2022 No. 104 of February 2, 2022 of Italian Ministry of University and Research; Project 2022S97PMY Structures for Quivers, Algebras and Representations (SQUARE).

\end{document}